\documentclass{article}
\usepackage[paper=a4paper,left=26mm,right=26mm,top=28mm,bottom=28mm]{geometry}
\usepackage[ngerman, english]{babel}
\usepackage{mathtools}
\usepackage{amssymb}
\usepackage{amsthm}
\usepackage{dsfont}
\usepackage{subcaption}

\def\e0{{e_0}}
\def\d{{d}}
\def\ds1{{\mathds{1}}}
\def\calA{{\mathcal{A}}}
\def\calF{{\mathcal{F}}}
\def\calH{{\mathcal{H}}}
\def\calL{{\mathcal{L}}}
\def\Hsym{H_{\text{sym}}}
\def\adm{\calA(d,\e0,b,r_0)} 
\def\burg{\mathcal{B}}
\newcommand{\up}{\textup}
\let\phi\varphi

\def\dd{\;\!\mathrm{d}}
\newcommand{\N}{\mathbb{N}}
\newcommand{\R}{\mathbb{R}}
\newcommand{\Z}{\mathbb{Z}}
\numberwithin{equation}{section}

\DeclareMathOperator{\curl}{curl}
\DeclareMathOperator{\supp}{supp}

\newtheorem{Def}{Definition}[section]
\newtheorem{Thm}[Def]{Theorem}
\newtheorem{Lemma}[Def]{Lemma}
\newtheorem{Cor}[Def]{Corollary}
\newtheorem{prop}[Def]{Proposition}
\newtheorem{Rm}[Def]{Remark}

\title{A scaling law for a model of epitaxially strained elastic films with dislocations}
\author{L. Abel, J. Ginster and B. Zwicknagl\\
\it\normalsize
Humboldt-Universit\"at zu Berlin, Institut f\"ur Mathematik, Unter den Linden 6, 10099 Berlin
}
\date{}
\begin{document}
\maketitle
\thispagestyle{empty}
\abstract{A static variational model for shape formation in heteroepitaxial crystal growth is considered. The energy functional takes into account surface energy, elastic misfit-energy and nucleation energy of dislocations. A scaling law for the infimal energy is proven. The results quantify the expectation that in certain parameter regimes, island formation or topological defects are favorable. This generalizes results in the purely elastic setting  from \cite{GZ:2013}. To handle dislocations in the lower bound, a new variant of a ball-construction combined with thorough local estimates is presented.}

\section{Introduction}
We prove a scaling law for the infimal energy in a variational model for heteroepitaxial growth introduced in \cite{FFLM:2016}.
Our aim is to understand the topological properties of quantum dots in a static situation.
Let us briefly explain the physical situation. 
We consider a crystalline film on a rigid crystalline substrate. 
A misfit between the corresponding lattice parameters introduces an elastic strain in the film.
To release elastic energy in the film the model under investigation allows for the presence of dislocations.    
Then pattern formation is often explained as a result of a competition between the surface energy of the film, elastic energy, and the dislocation nucleation energy, see \eqref{eq:totalenergy}. 
Before discussing the model, its physical motivation, and related literature in detail (see Section \ref{sec:model}), let us state our main result. 
\subsection{Main result}
Given the surface tension $\gamma >0$, the volume $d>0$, the lattice misfit strength $\e0>0$, the Burgers vector $(b,0)$ with $b>0$, the dislocation core radius $r_0>0$, and a typical linear elastic energy density $W$ satisfying a quadratic growth condition of the form
\begin{align} \label{eq:Wgrowth}
	c_1 \vert \Hsym\vert^2 \leq W(H) \leq \frac{1}{c_1}\vert \Hsym\vert^2
\end{align} 
we consider the energy functional $\calF: \adm \rightarrow \R$ given by 
\begin{align}\label{eq:totalenergy}
	\calF(h,H,\sigma)&
 \coloneqq \gamma \int_0^1 \sqrt{1+\vert h'\vert^2} \dd \calL^1 + \int_{\Omega_h} W(H) \dd \calL^2 + kb^2.
\end{align}
The three terms of the energy are the surface energy, the elastic energy and the dislocation nucleation energy.
The set of admissible configurations consists of triples of film profiles, elastic strains, and dislocation measures
\begin{align}\label{eq:adm}
	\calA\coloneqq& \adm \coloneqq \Big\{(h,H,\sigma)\mid h:[0,1]\to[0,\infty) \text{\ Lipschitz}, \int_0^1 h\dd\calL^1=d, \ h(0)=h(1)=0; \nonumber\\
		&\sigma =(b,0)\sum_{i=1}^k\delta_{p_i}\text{ such that }k\in\N_0, \, B_{r_0}(p_i)\subseteq\Omega_h;\text{\ and\ }\nonumber\\
		&H\in L^2(\Omega_h;\R^2)\text{ such that }\curl H=\sigma \ast J_{r_0} \up{ in } \Omega_h \up{ and } H(1, 0) = (\e0,  0) \up{ on } \{y=0\} \Big\},
\end{align}
where 
\begin{eqnarray}\label{eq:Omegah}
  \Omega_h\coloneqq \left\{(x,y)\in \R^2 \mid 0<x<1, \,0<y<h(x)\right\},  
\end{eqnarray} $J_{r_0}$ is a mollifying kernel with support in $B_{r_0}(0)$ and the curl-operator acts row-wise.
In order to understand qualitative properties of the small energy configurations we study the scaling of the infimal energy with respect to the problem parameters. 
Our analysis indicates that in certain parameter regimes equidistantly distributed dislocations occur close to the interface in isolated islands or in flat films, see the discussion in Section \ref{heuristics}.
Precisely, we have the following result.
\begin{Thm}[\up{Scaling Law}] \label{thm: main_intro}
There is a constant $c_s>0$ with the following property: For all 
$\gamma, \e0,b,\d>0$ and $r_0\in (0,1]$ with  $b/\e0 \geq 64^4 r_0$ there holds
\begin{eqnarray*}
   \frac{1}{c_s} s(\gamma, e_0,b,d,r_0) 
  \leq \inf_{\adm} \calF(h,H,\sigma)
  \leq  c_s s(\gamma, e_0,b,d,r_0), 
\end{eqnarray*}
where $s(\gamma, e_0,b,d,r_0) =  \gamma(1 + \d) + \min\left\{\gamma^{2/3}\e0^{2/3} \d^{2/3}, \left[\gamma \e0b\d\left(1+\log\left(\frac{b}{\e0r_0}\right)\right)\right]^{1/2}\right\}$.
\end{Thm}
\begin{proof}
We prove the upper bound in Section \ref{sec:upperbound} (see Cor. \ref{cor:ub}) and the lower bound in Section \ref{sec:lowerbound}.
\end{proof}
Theorem \ref{thm: main_intro} on the one hand-side complements the existence results of \cite{FFLM:2016} for energy functionals similar to \eqref{eq:totalenergy} by a scaling law for the infinitesimal energy, and on the other hand-side generalizes \cite[Theorem 3.2]{GZ:2013} for $p=2$  where the special case of $\gamma=1, \sigma=0$ and $k=0$ is treated. The main novelties here are an upper bound construction including dislocations and the proof of the lower bound.
Let us also briefly comment on the specific situation of flat films as discussed, for example, in \cite{Matthews,Haataja:2002}. 
There it is shown for a flat film of length $1$ that the presence of dislocations is favorable if $d \lesssim \frac{b}{\e0} \log(d/r_0)$. 
Using the upper bound constructions in this manuscript in the situation of flat films one can essentially validate this result, see Section \ref{sec: flat}.
However, we consider in this work not only flat films  which makes the situation significantly more complex.

\subsection{The model}\label{sec:model}
As indicated above, we study the model \eqref{eq:totalenergy} introduced in \cite{FFLM:2016}. The latter builds on a large body of literature on variational models that have proven useful to explain corrugations or island formation in epitaxially grown films as result of a competition between elastic and surface energy (see e.g. \cite{BonCha,FFLMor:2007,GZ:2013,FPZ:14,BGZ:2015}). Such models take the form \eqref{eq:totalenergy} on admissible configurations  in $\adm$ with $k=0$, i.e., $\sigma=0$. Besides existence and qualitative properties such as energy scaling laws, refined results have been obtained in particular in the static case (see e.g. \cite{FusMor,Bonac,Bonac:2015,DavPio:2019,DP:2020,CF:20} and the references therein), for related dynamical problems (see e.g. \cite{piovano:2014, FJM20, piovano-sapio:2023} and the references therein), microscopic justifications have been provided (see e.g. \cite{Kreutz-Piovano:21}), and the relations between nonlinear and linear elastic models have been studied (see e.g. \cite{FKZ:21}).\\

\noindent Let us briefly explain the model. The substrate is assumed to occupy the domain $(0,1)\times(-\infty,0)$, and the film the domain $\Omega_h$ (see \eqref{eq:Omegah}), where the profile function $h$ describes the film's free surface. \\

\noindent{\bf Surface energy.} The first term in \eqref{eq:totalenergy} then models the surface energy of the film's free surface, where $\gamma>0$ denotes a typical surface energy constant. Note that for $\gamma>0$, there is no configuration with vanishing surface energy. \\

\noindent{\bf Elastic energy. }The second term in \eqref{eq:totalenergy} models the elastic energy in the film. In this case, $H\in L^2_{\text{loc}} (\Omega_h;\R^{2\times 2})$ is the displacement field, and $W$ is a typical linearized elastic energy density, e.g.
\[W(H)=\mu|\Hsym|^2+\frac{\lambda}{2}(\text{tr\,}H)^2 \]
with the Lam\'{e} coefficients $\lambda,\mu$ fulfilling the ellipticity conditions $\mu > 0$ and $\mu + \lambda >0$.  
The crystallographic misfit between substrate and film is introduced via the Dirichlet boundary condition $H(1,0)=(\e0,0)$ on $\{y=0\}\cap\overline{\Omega_h}$ for admissible configurations. Note that this is well-defined since every admissible strain field has curl in $L^2(\Omega_h;\R^2)$ (see \eqref{eq:adm} and the discussion around \eqref{eq:HcurlL2} below) and therefore admit a tangential trace, c.f. \cite[Chapter IX., Part A, Theorem 2]{Lions3} or \cite[Chapter 4]{Boyer/Fabrie:2012}.
\\
Note that for $\e0>0$, there is no configuration with vanishing elastic energy. \\

\noindent

\noindent {\bf Dislocations and nucleation energy. }A competing method for strain relief that is observed in experiments, is the development of topological defects such as dislocations (see e.g. \cite{Tersoff-LeGoues:94,Gao-Nix:99,Haataja:2002}). This effect is modeled in \cite{FFLM:2016} in terms of the dislocation measure $\sigma=\mathcal{B}\sum_{i=1}^k\delta_{p_i}$ for finitely many dislocation centers $p_i\in \Omega_h$. Here, we take the lattice structure in the film as reference configuration, and follow the Volterra approach to view dislocations as topological defects (see e.g. \cite{Nabarro}).  For simplicity, we restrict ourselves to the case of only one Burgers vector $\burg=(b,0) \in \R^2$, and denote by  $r_0\in (0,1]$ the dislocation radius. In view of \cite{FFLM:2016}, the important assumption appears to be that the first component $b$ of the Burgers vector has the same sign as the crystallographic misfit parameter $\e0$. We take the second component of $\burg$ as $0$ for simplicity. Note that we restrict ourselves to the case $r_0\leq 1$ since otherwise there is no non-vanishing admissible dislocation measure as we assume that the balls $B_{r_0}(p_i)$ around the dislocation centers are completely contained in $\Omega_h$. \\
Since in the continuum theory dislocations correspond to singularities in the strain field, some regularization is required. We follow the convolution-based approach from \cite{FFLM:2016} and consider for the core radius $r_0>0$ a mollifier $J_{r_0}(x) = r_0^{-2} J_1(x/r_0)$, where $J_{1}\in C_c^\infty(B_{1}(0); [0,\infty))$ satisfies $\int_{\R^2}J_{1}\d\calL^2=1$. This results in the condition
\begin{eqnarray}\label{eq:HcurlL2}
    \curl H=\sigma\ast J_{r_0},
\end{eqnarray}
where the curl-operator is applied row-wise. In the analogous sense we will refer to $\operatorname{curl} H_1$ to be the $\operatorname{curl}$ of the first row of the matrix field $H$. Here and in the remainder of the text, we do not distinguish in notation between row and column vectors. 
The nucleation energy associated to $\sigma=\burg\sum_{i=1}^k \delta_{p_i}$ is then given by the third term in \eqref{eq:totalenergy}, more precisely, 
\begin{eqnarray}
 \label{eq:nucleationenergy}
 c_0kb^2.
\end{eqnarray}
Here, the parameter $c_0>0$ is a material constant. The nucleation energy should represent the
core energy of a dislocation. So heuristically, one could compute for a single dislocation at point $p$
the elastic energy (up to a Korn constant) via
\begin{align*}
    \int_{B_{r_0}(p)} \vert H\vert^2 \dd \calL^2 &\geq \int_0^{r_0} \int_{\partial B_t(p)} \vert H\cdot \tau \vert^2 \dd \calH^1 \dd \calL^1(t)\\
        &\geq \int_0^{r_0} \frac{1}{2\pi t} \left\vert \int_{\partial B_t(p)} H\cdot \tau \dd \calH^1 \right\vert^2 \dd \calL^1(t)\\
        &= \int_0^{r_0} \frac{1}{2\pi t} \left\vert \int_{B_t(p)} b(\delta_p \ast J_{r_0}) \dd \calL^2 \right\vert^2 \dd \calL^1(t)\\
        &=\int_0^{1} \frac{\vert b\vert^2}{2\pi t} \left\vert \int_{B_t(0)} J_{1} \dd \calL^2 \right\vert^2 \dd \calL^1(t)\\
        &=:\vert b\vert^2 C(J_1),
\end{align*}
where we denote by $\tau$ a tangent unit vector field.
\noindent Note that $C(J_1) \coloneqq \int_{0}^{1} \frac{1}{2\pi t} \left\vert \int_{B_t(0)} J_{1} \dd \calL^2 \right\vert^2 \dd \calL^1(t)$ exists as $0<\frac{1}{2\pi t}\left\vert \int_{B_t(0)} J_{1} \dd \calL^2 \right\vert^2 \leq c t^3.$ Thus the material constant $c_0$ can be treated as a constant of order one and is not a parameter. For simplicity of notation, we restrict ourselves to the case  $c_0 = 1$.
In the upper bound, we make the dependency on $c_0$ explicit, see Remark \ref{rem:ubc0}.\\

\subsection{Presence of Dislocations in Flat Films} \label{sec: flat}

In this section, we will consider the specific situation of a flat film of length $L$ and height $d/L$. 
As in this setting the surface energy is fixed, in the following we will compare only the elastic energies for the two constructions of the elastic strain $H$ that will be used in a slightly different form in the proof of the upper bound of Theorem \ref{thm: main_intro}.
First, let us consider the elastic strain $H = \nabla u$, where $u: (0,1) \times (0,d) \to \R^2$ is given by
\[
u(x,y) = \begin{cases} (\e0 x \frac{L - y}L,0) &\text{ if } y \leq L, \\ 
    (0,0) &\text{ else. }
\end{cases}
\]
It follows that $\int_{(0,L) \times (0,d/L)} W(H) \, d\mathcal{L}^2 \sim \min\{\e0^2 L^2, \e0^2 d\}$, corresponding to the cases $L \leq d/L$ and $L \geq d/L$, respectively. 

If dislocations are present they can be expected to occur at distance $b/\e0$ at the interface to compensate the elastic strain induced by the misfit.
Assuming that $L \leq b / \e0$, let us consider a configuration of $\sim L \frac{\e0}b$ equidistant dislocations with Burgers vector $(b,0)$ and distance $b/\e0$. 
A corresponding strain field $H: (0,L) \times (0,d/L) \to \R^{2 \times 2}$ can be constructed such that the elastic energy is essentially the sum of the self-energies of the different dislocations, i.e.,
$\int_{(0,L) \times (0,d/L)} W(H) \d \mathcal{L}^2 \sim L \frac{\e0}b \min\{ b^2 \log( d / (L r_0) ), b^2 \log ( b / (\e0 r_0)) \}$, yielding two different scaling regimes, corresponding to $d / L \leq b/\e0$ and $d / L \geq b / \e0$, respectively, c.f.~the proof of Proposition \ref{prop:ub}. 

Comparing the different elastic energies suggests that the presence of dislocations is energetically favorable in flat films iff it holds $\min\{ L, d/L \} \gtrsim \frac{b}{\e0} \log( b / (\e0 r_0) )$. The estimate for the height validates the findings from \cite{Matthews,Haataja:2002} (up to a refinement of order $\log ( \log (d/L))$).

\subsection{Heuristics of the proof}\label{heuristics}
Before giving the detailed proof, let us briefly explain the scaling law in Theorem \ref{thm: main_intro}. 

The term $\gamma(1+\d)$ in the scaling law simply follows from the surface energy (see \cite[Lemma 2.6]{GZ:2013}). 
Indeed, let $h$ be  any admissible profile and denote by $\overline{x}\in (0,1)$ a point where $h$ attains its maximum (this exists since $h$ is continuous). 
Then $h(\overline{x})\geq \d$ and consequently
\begin{eqnarray}\label{eq:surface}
  \gamma \int_0^1\sqrt{1+|h'|^2}\dd \calL^1\geq \frac{\gamma}{2}\int_0^1\left(1+|h'|\right)\dd \calL^1= \frac{\gamma}{2}
+\frac{\gamma}{2}\left(\int_0^{\overline{x}}|h'|\dd\calL^1+\int_{\overline{x}}^1|h'|\dd\calL^1\right)\geq \frac{\gamma}{2}+\gamma\d,\end{eqnarray}
and this estimate is sharp up to a constant for configurations as in Figure \ref{fig:profile}. The remaining term $\min\left\{\gamma^{2/3}\e0^{2/3} \d^{2/3}, \left[\gamma \e0b\d\left(1+\log\left(\frac{b}{\e0r_0}\right)\right)\right]^{1/2}\right\}$ in the scaling law reflects the competition of surface, elastic and dislocation nucleation energy in configurations as sketched in Figure \ref{fig:profile}. 
More precisely, if there are no dislocations, then an island of length $L\in (0,1]$ has height $\sim d/L$, and the elastic energy is estimated in terms of the trace norm of the Dirichlet boundary condition, which leads to an energy (c.f. also the argument in Section \ref{sec: flat} above)
\[
    \sim \gamma+ \frac{\gamma d}{L}+\e0^2 L^2
\]
and in particular a natural length scale $L\sim \min\{\e0^{-2/3}(\gamma d)^{1/3},1\}$. If $L=1$ then $\e0^2\leq \gamma d$, and hence in any case, the energy is estimated above  by 
\begin{eqnarray*}
    \lesssim \gamma(1+d)+\gamma^{2/3}\e0^{2/3} \d^{2/3}.
\end{eqnarray*}
Note that the case $L<1$ corresponds to the formation of isolated islands.\\
As mentioned already in Section \ref{sec: flat}, dislocations are expected to occur at distance $l=b/\e0$ at the interface to compensate the elastic strain induced by the misfit. A configuration as sketched in Figure \ref{fig:profile} then has roughly $L\e0/b$ dislocations, and its total energy is estimated by
\[\lesssim \gamma+ \frac{\gamma d}{L}+\frac{L\e0}{b}b^2\log\left(\frac{b}{\e0r_0}\right)+\frac{L\e0}{b}b^2,\]
where the log-term represents the self-energy of the dislocations. Optimizing in $L$ yields $$L\sim \min\{(\gamma d)^{1/2} [\e0 b (\log(b/(\e0r_0))+1)]^{-1/2}, 1\},$$
which leads to an upper bound for the energy of the form
\[ \lesssim \gamma(1+d)+\left[\gamma \e0b\d\left(1+\log\left(\frac{b}{\e0r_0}\right)\right)\right]^{1/2}.\]
To prove the lower bound, we introduce similarly to \cite{BGZ:2015} local length scales as sketched in Figure \ref{fig:squares}. The idea then is to use on each of the segments of length $\ell_i$ with local volume $d_i$ the following lower bounds of the energy:
 \begin{itemize}
     \item if $d_i\gg\ell_i^2$ then the surface energy can be bounded below by $\sim \ell_i+\frac{d_i}{\ell_i}$
     \item if the number of dislocations is greater than $\ell_i\e0/b$ then we bound the elastic energy below by $\sim \frac{\ell_i \e0}{b}b^2\log\left(\frac{b}{\e0 r_0}\right)$ 
     \item if the number of dislocations is smaller than $\ell_i\e0/b$ then the elastic energy behaves roughly as in the case without any dislocations, and we bound the elastic energy below by $\sim \e0^2\ell_i^2$.
 \end{itemize}
 Using interpolation estimates as sketched above and subadditivity in $d_i$, we would obtain the lower bound. \\
 Technical difficulties arise due to the fact that all estimates on the elastic energy use Korn-type inequalities. To overcome this, we use a version of a ball construction which is an established tool to prove lower bounds for the self-energy. This requires slightly different local length scales than one would expect.  

\begin{figure}[h]
        \centering
        \includegraphics[height=4cm]{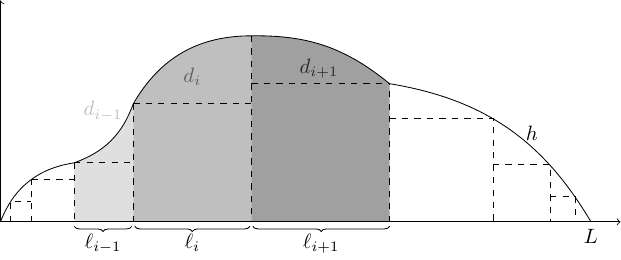}
        \caption{sketch of local length scales $\ell_i$ with (gray-shaded) local volumes $d_i$ }
\label{fig:squares}
\end{figure} 
\section{Upper Bound}\label{sec:upperbound}
In this section, we prove the following proposition which implies the upper bound in Theorem \ref{thm: main_intro}. Our new contribution is the proof of the first part. The second part follows from \cite[Theorem 3.1]{BGZ:2015}. 
\begin{prop}\label{prop:ub}
There is a constant $c_s>0$ with the following property: For all  $\e0,b, \gamma,\d,r_0>0$  the following two assertions hold:
\begin{itemize}
	\item[(i)] If $r_0\in (0,1]$ and $4r_0\e0\leq b$, then there is a configuration $(h,H,\sigma) \in \adm$ such that the total energy is bounded by 
\begin{align*}
	\calF(h,H,\sigma)\leq c_s\left(\gamma(1 + \d) + \left[\gamma \e0b\d\left(1+\log\left(\frac{b}{
 \e0r_0}\right)\right)\right]^{1/2}\right).
	\end{align*}
 \item[(ii)] There is a configuration $(h,H,\sigma) \in \adm$ such that the total energy is bounded by
\begin{align*}
	\calF(h,H,\sigma)\leq c_s\left(\gamma(1+\d)+\left(\e0 \gamma\d\right)^{2/3}\right).
	\end{align*}
\end{itemize}
\end{prop}
\begin{Rm}\label{rem:ubc0}
If we replace the nucleation energy term (last term) in \eqref{eq:totalenergy} by the term \eqref{eq:nucleationenergy} involving the material parameter $c_0$, which we expect to be of order one, we can also make the dependences in this parameter explicit. Carefully checking the proof, we find that there is a constant $c_s$ independent of $c_0$ such that (ii) holds, and under the assumptions of (i) there is an admissible  configuration with 
\[
    \calF(h,H,\sigma)\leq c_s\left(\gamma(1 +\d) + \left[\gamma \e0b\d\left(1+ c_0 + \log\left(\frac{b}{\e0r_0}\right)\right)\right]^{1/2}\right).\]

\end{Rm}
   \begin{proof}
We start with (i).
Let $\e0,b, \gamma,\d,r_0>0$ be such that $r_0\in (0,1]$ and $4r_0\e0\leq b$. We explicitly construct an admissible triple $(h,H,\sigma)\in\adm$.
\\

{\it Step 1: Geometry and surface energy. }
 We use the same profile function as in \cite[Theorem 3.1]{BGZ:2015}. Let $0<L\leq \min\{1,d/(4r_0)\}$, to be fixed later (see \eqref{eq:choiceLub}). Set $\delta:=\frac{1}{16}\min\{r_0,L\}$ and $\overline{h}\coloneqq \frac{\d}{L-\delta}$, and consider (see Figure \ref{fig:profile}) $h:[0,1]\to[0,\infty)$, 

\begin{equation}\label{eq:ubchoiceh}
        h(x) \coloneqq \begin{cases}
                        \frac{\overline{h}}{\delta}x &\text{if } 0\leq x \leq \delta, \\
                        \overline{h} &\text{if }\delta \leq x \leq L-\delta, \\
                        -\frac{\overline{h}}{\delta}x + \frac{\overline{h}}{\delta}L &\text{if } L-\delta\leq x \leq L, \\
                        0 &\text{if } L \leq x\leq 1. 
                    \end{cases}
\end{equation}

\begin{figure}[t]
        \centering
        \includegraphics[height=5cm]{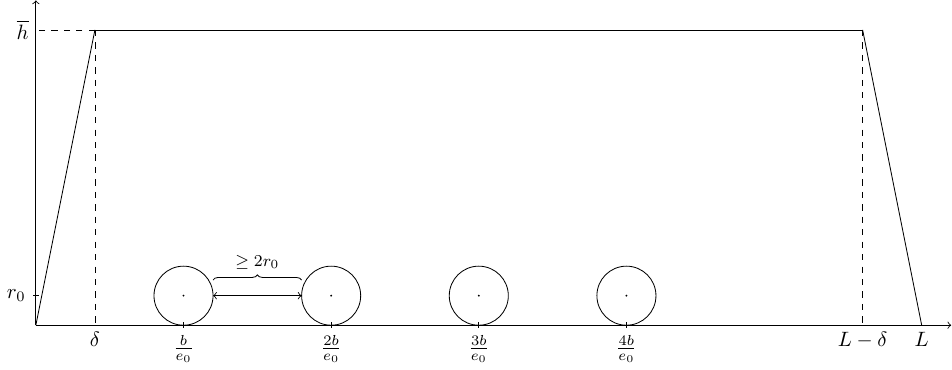}
        \caption{profile $h$ with indicated dislocations for the upper bound of the scaling law (not to scale) }
\label{fig:profile}
\end{figure}
Then $h$ is an admissible film profile, i.e., $h: [0,1] \to [0,\infty)$ is Lipschitz, $h(0) = h(1) = 0$ and $\int_0^1 h \, d \mathcal{L}^1 = d$. Moreover, the surface energy can be estimated by
\begin{align}\label{eq:ubsurf}
   \gamma  \int_0^1 \sqrt{1+\vert h'\vert^2}\dd \calL^1 \leq \gamma\int_0^1\left(1+|h'|\right)\,\dd\calL^1\leq \gamma(1+2\overline{h})\leq \gamma+\frac{4\gamma \d}{L}.
\end{align}

 {\it Step 2: Dislocations and nucleation energy. }
We set $k\coloneqq \lfloor \frac{L\e0}{b}\rfloor\in\N_0$ and consider $\max\{k-1,0\}$ equidistant dislocations, see Figure \ref{fig:profile}, i.e., we set (recall that $\burg=(b,0)$)
\[\sigma:=\burg\sum_{i=1}^{k-1}\delta_{(z_i,r_0)}\text{\qquad with\quad}z_i:=\frac{ib}{\e0},\quad  i=1,\dots, \max\{0,k-1\}. \]
 Then we have $B_{r_0}((z_i,r_0))\subseteq \Omega_h$ for all $i=1,\dots, k-1$. Indeed,  $B_{r_0}((ib/\e0,r_0))\subseteq (i b/\e0-r_0,i b/\e0+r_0)\times(0,2r_0)\subseteq\Omega_h$ since $b/\e0\geq 4r_0$, $b/\e0-r_0>3r_0>\delta$, $(k-1) b / \e0 + r_0 \leq L - b/\e0 + r_0 \leq L - \delta$ and $\overline{h}\geq d/L\geq 4r_0$. Thus, $\sigma$ is an admissible dislocation measure, and the nucleation energy is estimated by 
\begin{eqnarray}
\label{eq:ubnucl}
        c_0 (k-1) b^2 \leq c_0 \e0 b L .
\end{eqnarray}
{\it Step 3: Construction of the strain field. }It remains to construct a strain field $H$. 
We consider first the main part of the domain, i.e., $((0,(k-1)b/\e0]\times\R)\cap\Omega_h$. This is the part in which dislocations occur. We use a periodic construction and
set \begin{eqnarray*} 
A&\coloneqq& \left\{ (x,y)\in \R^2 \mid x\in (0,b/\e0], y\in (-\infty,r_0\e0x/b]\right\},\\
B&\coloneqq& \left\{ (x,y)\in \R^2 \mid x\in (0,b/\e0], y\in (r_0\e0x/b, (r_0\e0/b-1)x+b/\e0)\right\},\text{\quad and}\\
C&\coloneqq& ((0,b/\e0]\times\R)\setminus (A\cup B).
\end{eqnarray*}
We define the function $\hat{M}\in L^1_{\text{loc}}(\R^2;\R^{2\times 2})$ to be $(b/\e0,0)$-periodic, i.e., $\hat{M}(x+b/\e0,y)=\hat{M}(x,y)$ for all $(x,y)\in \R^2$, and such that its restriction to $(0,b/\e0) \times \R$ is given as 
the gradient field $\hat{M}|_{(0,b/\e0) \times \R}:=\nabla u$ with $u:(0,b/\e0) \times \R\to\R^2$ given by (see Figure \ref{fig:construction_detail})
\[
        u(x,y) \coloneqq
        \begin{cases}
                (\e0x,0) & \text{for } (x,y) \in A, \\
                \left( \frac{\e0x\left[\left(\frac{r_0\e0}{b}-1\right)x+\frac{b}{\e0}-y\right]}{\left(\frac{b}{\e0}-x\right)},0\right) & \text{for } (x,y) \in B, \\
                (0,0) & \text{for } (x,y) \in C. 
        \end{cases}
\]
For upcoming estimates note that in $B$ it holds
\begin{align} \label{eq: formula M}
\hat{M}_{11} &= \frac{\left[2\e0x\left(\frac{r_0\e0}{b}-1\right)+b-\e0y\right]\left(\frac{b}{\e0}-x\right)+\e0x^2\left(\frac{r_0\e0}{b}-1\right)+bx-\e0xy}{\left(\frac{b}{\e0}-x\right)^2} 
\end{align}
and 
\begin{align} \label{eq: formula M2}
\hat{M}_{12} &= -\frac{\e0 x}{b/ \e0 - x}.
\end{align}

\begin{figure}[h]
        \centering
        \includegraphics[scale=0.35]{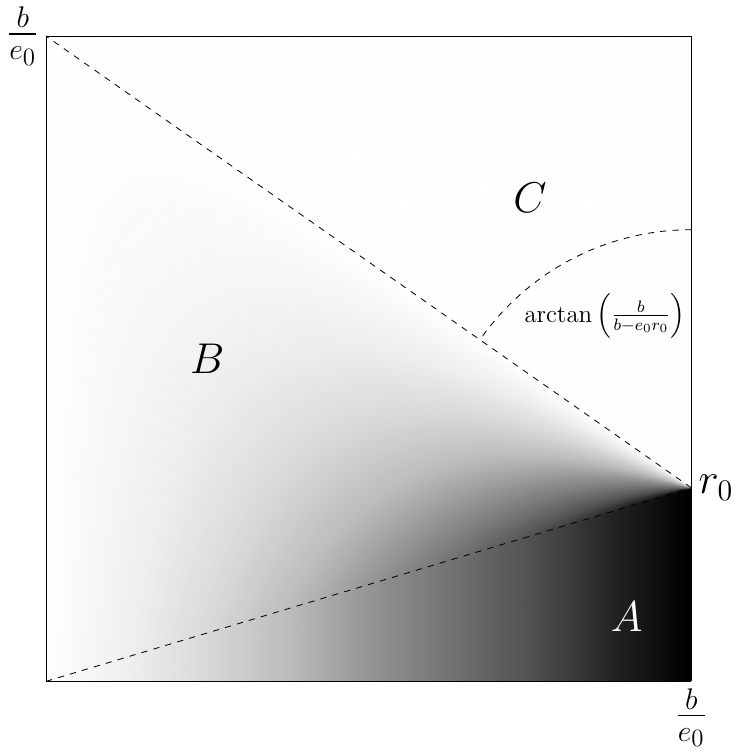}
        \caption{detail of the construction of the displacement $u_1$, the greyscale represents $u_1(x,y)$, white represents $u_1(x,y)=0$, black represents $u_1(x,y)=b$.}
        \label{fig:construction_detail}
\end{figure}
Note that $u$ satisfies $\lim_{x\to 0}u(x,y)=(0,0)$ for all $y\in\R$, and 
\[u(b/\e0,y)=\begin{cases}
    (b,0)&\text{\quad if\ }y\leq r_0\\
    (0,0)&\text{\quad if\ }y>r_0,
\end{cases} \]
and its $(b/\e0,0)$-periodic extension is locally in $SBV$ with $D^J u = \sum_{i \in \Z} (b,0) \otimes e_1 \, \mathcal{H}^1_{|\{i b/\e0\} \times (-\infty,r_0)}$. 
 In the most right part of the film (where no dislocations occur), we use a slightly different construction. We define $\hat{N}:((k-1)b/\e0,\infty)\times\R\to\R^{2\times 2}$ as $\hat{N}(x,y):=\nabla v(x,y)$ with 
 \[v(x,y):=\begin{cases}
      \left( \e0(x-(k-1)b/\e0) ,0\right) &\text{\ if\ }y\leq 0,\\
     (\e0\left(x-(k-1)b/\e0-y\right),0)& \text{\ if \ }y\in(0, x-(k-1)b/\e0),\\
     (0,0)& \text{\ otherwise.}
 \end{cases}\]
 Note that it also holds $\lim_{x \searrow (k-1) b / \e0} v(x,y) = (0,0)$.
Finally set $\hat{H}\in L^2_{\text{loc}}(\R^2;\R^{2\times 2})$,
\begin{eqnarray*}
\hat{H}(x,y):=\begin{cases}
    \hat{M}(x,y)&\text{\ if\ }x\leq (k-1)b/\e0,\\
    \hat{N}(x,y)&\text{\ if\ }x>(k-1)b/\e0.
\end{cases}
\end{eqnarray*}
Then $\operatorname{curl} \hat{H} = \sigma$.
Now, set $H\coloneqq (\hat{H}\ast J_{r_0})\mid_{\Omega_h}$. Then 
$(h,H,\sigma) \in \adm$ is an admissible configuration.\\

{\it Step 4: Estimates for $H$.}
We first claim that it holds for $(x,y) \in \R^2$ with $x \leq L + b / \e0$
\begin{equation} \label{eq: est hat H}
|\hat{H}(x,y)| \leq C \left( \frac{b}{\operatorname{dist}((x,y), \{ (i b/ \e0, r_0): i \in \Z \})} \ds1_{\{ y \leq b / \e0\}} + \e0 \ds1_{\{y \leq 3 b / \e0 \}}  \right).
\end{equation}
First, note that $L + b / \e0 - (k-1) b / \e0 = L - (k-2) b / \e0 \leq 3 b/\e0$.
Hence, \eqref{eq: est hat H} holds by the definition of $\hat{N}$ if $x > (k-1) b/ \e0$, i.e., $\hat{H}(x,y) = \hat{N}(x,y)$.
Hence, it remains to show the above estimate when $\hat{H}(x,y) = \hat{M}(x,y)$. 
By periodicity of $\hat{M}$ we may assume that $(x,y) \in (0,b/\e0] \times \R$.
Then note that for $(x,y) \in A$ that $|\hat{H}(x,y)| \leq \e0$ whereas for $(x,y) \in C$ it holds $|\hat{H}(x,y)| = 0$. 
This implies the validity of \eqref{eq: est hat H} in $A \cup C$. 
Eventually, let $(x,y) \in B$.
Then we estimate using \eqref{eq: formula M}
\begin{align*}
    |\hat{H}_{11}(x,y)| \leq &\frac{ \left| 2 \e0x  \frac{r_0 \e0}b \right| + \e0 \left| 2 x \right| + b + \e0 \left|  y \right|}{| b/\e0 - x|} + \frac{\e0x \left( \frac{r_0 \e0}{b} x  - x +  b / \e0 - y \right)}{|b/ \e0 - x |^2} \leq  \frac{7b}{|b/\e0 - x|},
\end{align*}
where we used that in $B$ it holds $ \frac{r_0 \e0}b x \leq y \leq \frac{r_0 \e0}b x - x + b/\e0$ and $0 \leq x \leq b/\e0$.
Similarly, we obtain from \eqref{eq: formula M2} that $|\hat{H}_{12}(x,y)| \leq \frac{b}{|b /\e0 - x|}$ using $0\leq x \leq b / \e0$.
Next, note that it holds for $(x,y) \in B$ by the definition of the set $B$ that
\begin{align*}
|y - r_0| \leq \left| \frac{r_0 \e0}b x - x + b/\e0 - r_0\right| + \left| \frac{r_0 \e0}b x - r_0\right|
\leq \left(2 \frac{r_0 \e0}b + 1\right) \left|x - \frac{b}{\e0}\right|   \leq 3 \left|x - \frac{b}{\e0}\right|
\end{align*}
and consequently
\begin{align*}
    |\hat{H}(x,y)| \leq C \frac{b}{|b/\e0 - x|} \leq 4C \frac{b}{|b/\e0 - x| + |y - r_0|} \leq 4C \frac{b}{\operatorname{dist}(x,\{(i b/\e0,r_0): i \in \Z\})}.
\end{align*}
This finishes the proof of \eqref{eq: est hat H}. 

Next, we show the following corresponding estimate for $H$
\begin{align} \label{eq: est H}
    |H(x,y)| \leq C \left(  \frac{b}{r_0} \ds1_{ S}(x,y) + \frac{b}{\operatorname{dist}((x,y),\{(ib/\e0,r_0): i \in \Z \})} \ds1_{\{ y \leq 4 b/\e0\} \setminus S}(y) + \e0 \ds1_{\{ y \leq 4 b/\e0\}}(y) \right),
\end{align}
where $S = \bigcup_{i \in \Z} B_{2r_0}((i b/\e0,r_0))$.
First, note that since $r_0 \leq b / \e0$ it follows from \eqref{eq: est hat H} that $H(x,y) = 0$ if $y \geq 4 b / \e0$.
Moreover, we obtain by \eqref{eq: est hat H} that for $y \leq 4 b / \e0$
\begin{align*}
    |H(x,y)| \leq &C\left( \left(\frac{b}{\operatorname{dist}(\cdot,\{(ib/\e0,r_0): i \in \Z\})} * J_{r_0}\right) (x,y) + (\e0 * J_{r_0} )(x,y) \right) \\
    \leq &C \left( \left(\frac{b}{\operatorname{dist}(\cdot,\{(ib/\e0,r_0): i \in \Z\})} * J_{r_0} \right)(x,y) + e_0 \right).
\end{align*}
First, let us consider $(x,y) \in \R^2 \setminus S$ and $(x',y') \in B_{r_0}(x,y)$. 
Then it holds 
\begin{align*}
\operatorname{dist}((x',y'), \{(ib/\e0,r_0): i \in \Z\}) \geq &\operatorname{dist}((x,y),\{(ib/\e0,r_0): i \in \Z\}) - r_0 \\ \geq &\frac12 \operatorname{dist}((x,y),\{(ib/\e0,r_0): i \in \Z\})
\end{align*}
and thus
\begin{align*}
    &\left(\frac{b}{\operatorname{dist}(\cdot,\{(ib/\e0,r_0): i \in \Z\})} \ast J_{r_0} \right) (x,y) \\
    = & \int_{B_{r_0}(x,y)} \frac{b}{\operatorname{dist}((x',y'), \{(ib/\e0,r_0): i \in \Z\})} J_{r_0}((x'-x,y'-y)) \, \dd x' \dd y' \\
    \leq &2 \frac{b}{\operatorname{dist}((x,y),\{(ib/\e0,r_0): i \in \Z\})}.
\end{align*}
On the other hand, let us now consider $(x,y) \in S$, i.e. there exists $i \in \Z$ such that $(x,y) \in B_{2r_0}((i b/ \e0, r_0))$.
Then we may estimate using Young's inequality for convolutions and $b/\e0 \geq 4 r_0$
\begin{align*}
\left| \left(\frac{b}{\operatorname{dist}(\cdot,\{(ib/\e0,r_0): i \in \Z\}} * J_{r_0} \right) (x,y) \right| 
\leq &\left\| \frac{b}{\operatorname{dist}(\cdot,\{(ib/\e0,r_0): i \in \Z\}} \right\|_{L^1(B_{r_0}((x,y))} \| J_{r_0} \|_{L^{\infty}} \\
\leq &C r_0^{-2} \int_{B_{3r_0}((ib/ \e0,r_0))} \frac{b}{|(x',y') - (ib/\e0,r_0)|} \, d\mathcal{L}^2 \\
= &6\pi C \frac{b}{r_0}. 
\end{align*}
This finishes the proof of \eqref{eq: est H}.

{\it Step 5: Estimate for the elastic energy.}

Using the estimate \eqref{eq: est H} we obtain by \eqref{eq:Wgrowth}
\begin{align} \nonumber
    &\int_{\Omega_h} W(H) \, d\mathcal{L}^2 \leq \frac1{c_1}\int_{\Omega_h}|\Hsym|^2\dd \calL^2\leq\frac{1}{c_1} \int_{\Omega_h}|H|^2\dd \calL^2 \\ \nonumber
    &\leq C \int_{(0,L) \times (0,4b/ \e0)} \left[ \frac{|b|^2}{r_0^2} \ds1_S + \frac{|b|^2}{\operatorname{dist}(\cdot ,\{(ib/\e0,r_0): i \in \Z\})^2} \ds1_{S^c} + \e0^2 \right]\, d\mathcal{L}^2 \\ \nonumber 
    &\leq C\left[ 4 \pi (k+2) b^2 + \sum_{i=0}^{k+1} \int_{B_{4b / \e0}((ib / \e0,r_0)) \setminus B_{2r_0}((ib / \e0,r_0))} \frac{b^2}{|(x,y) - (i b / \e0,r_0)|^2} \, \dd\mathcal{L}^2 (x,y) + 4 L b \e0  \right] \\ \nonumber
    &\leq C k b^2 (1 + \log(b/ (\e0 r_0)) + C L b \e0 \\ 
    &\leq C L b \e0 (1 + \log( b / (\e0 r_0)) ), \label{eq:ubelastic}
\end{align}
where we used that $k b^2 \leq L b \e0$ and
\begin{align*}
    &\int_{B_{4b / \e0}((ib / \e0,r_0)) \setminus B_{2r_0}((ib / \e0,r_0))} \frac{b^2}{|(x,y) - (i b / \e0,r_0)|^2} \, \dd\mathcal{L}^2 (x,y) \\ 
=& 2 \pi b^2 \int_{2r_0}^{4b / \e0} \frac1t \, dt = 2\pi b^2 \log(2 b / (\e0 r_0) ) \leq  2\pi b^2 ( \log( b / (\e0 r_0) ) + 1)  . 
\end{align*}

{\it Step 6: Choice of $L$ and conclusion. }
Combining \eqref{eq:ubsurf}, \eqref{eq:ubnucl}, and \eqref{eq:ubelastic}, we obtain
\begin{eqnarray}
    \label{eq:ubwithdisloc}
\calF(h,H,\sigma)\leq
        c\left( \gamma + \frac{\d\gamma}{L}+L\e0b\left(1+c_0+\log(b/(\e0r_0)\right)\right).
        \end{eqnarray}
We now choose 
\begin{eqnarray}
    \label{eq:choiceLub}
    L:=\min\left\{\gamma^{1/2}\d^{1/2}\left[\e0b(1+c_0+\log(b/(\e0r_0))\right]^{-1/2},\,1,\,\d/(4r_0)\right\}.
\end{eqnarray}
Note that this choice yields the minimum in the upper bound from \eqref{eq:ubwithdisloc}. We consider the three cases for the choice of $L$ in \eqref{eq:choiceLub} separately:\\
If $L=\left(\gamma\d\right)^{1/2}\left[\e0 b(1+c_0+\log(b/(\e0r_0))\right]^{-1/2}$, then inserting this choice in \eqref{eq:ubwithdisloc}
 yields the upper bound 
 \begin{eqnarray}\label{eq:ub1}
        \calF(h,H,\sigma) \leq c\left( \gamma(1 + \d) + \left[\gamma\e0b\d(1+c_0+\log(b/(\e0r_0))\right]^{1/2}\right).
\end{eqnarray}
If $L= 1$ then $1\leq (\gamma\d)^{1/2}\left[\e0b(1+c_0+\log(b/(\e0r_0))\right]^{-1/2}$, and thus \eqref{eq:ubwithdisloc} gives
\begin{eqnarray}\label{eq:ub2}
  \calF(h,H,\sigma) &\leq& c\left(\gamma(1+\d)+  (\gamma\d)^{1/2}\left[\e0b(1+c_0+\log(b/(\e0r_0))\right]^{-1/2}\e0b\left(1+c_0+\log(b/(\e0r_0)\right)\right)\nonumber\\
  &=& c\left( \gamma(1 + \d) + \left[\gamma\e0b\d(1+c_0+\log(b/(\e0r_0))\right]^{1/2}\right).
\end{eqnarray}
If $L=\d/(4r_0)$, we use $L=\d/(4r_0)\leq (\gamma\d)^{1/2}\left[\e0b(1+c_0+\log(b/(\e0r_0))\right]^{-1/2}$ and $r_0\leq 1$ to get from \eqref{eq:ubwithdisloc}
\begin{eqnarray}\label{eq:ub3}
\calF(h,H,\sigma) &\leq& c\left( \gamma(1 + \d) + 4\gamma r_0+\left[\gamma\e0b\d(1+c_0+\log(b/(\e0r_0))\right]^{1/2}\right)\nonumber\\
&\leq& c\left( \gamma(1 + \d) + \left[\gamma\e0b\d(1+c_0+\log(b/(\e0r_0))\right]^{1/2}\right). \end{eqnarray}
Combining \eqref{eq:ub1}, \eqref{eq:ub2}, and \eqref{eq:ub3} concludes the proof of the first part of Prop. \ref{prop:ub}. \\

Part 2 follows from the proof of the upper bound in \cite[Theorem 3.1]{GZ:2013}. For the readers' convenience, we briefly recall the construction, for details we refer to \cite[Theorem 3.1]{GZ:2013}. Let $L\in (0,1]$, and $h$ as in \eqref{eq:ubchoiceh}, see Figure \ref{fig:profile}. 
Let $H$ be the restriction to $\Omega_h$ of $\tilde{H}\colon [0,1]\times [0,\infty) \rightarrow \R^{2\times 2}$ given by
\[
        \tilde{H}(x,y) \coloneqq \begin{pmatrix} \e0(1-\frac{1}{L}y) & -\frac{1}{L} \e0x \\
                0 & 0
        \end{pmatrix}
\]
for $y\in [0,L]$ and $\tilde{H}(x,y) \equiv 0$ for $y>L$. Note that $H=\nabla u$ for $u$ as defined in the proof of \cite[Theorem 3.1]{GZ:2013}. Set $\sigma = 0$. Then $(h,H,\sigma)\in \adm$, and (recall \eqref{eq:ubsurf})
\[
    \calF(h,H,\sigma) \leq 4\gamma\left(1+\frac{\d}{L}\right)+\frac{1}{c_1}\int_{\Omega_h}\vert H\vert^2\dd\calL^2+0 \leq c\left(\gamma+\frac{\d\gamma}{L}+2\e0^2 L^2\right).
\]
If $\e0^2 \leq \d\gamma$, then we choose $L\coloneqq 1$, and if $\e0^2 \geq \d\gamma$, we choose $L\coloneqq (\d\gamma)^{1/3}\e0^{-2/3} \leq 1$.
It follows in both cases (using $\e0^2\leq (\e0\gamma \d)^{2/3}$ in the first case) that
\[
	\calF(h,H,0) \leq c \left(\gamma(1 + \d) + (\gamma \e0\d)^{2/3}\right).
\]
\end{proof}
We note that Proposition \ref{prop:ub} implies the upper bound in Theorem \ref{thm: main_intro}.
\begin{Cor}\label{cor:ub}
    There is a constant $c_s>0$ with the following property: For all 
$\e0,b,\d>0$ and $r_0\in (0,1]$ with  $b/\e0 \geq 4 r_0$ there holds
\begin{eqnarray*}
  \inf_{\adm} \calF(h,H,\sigma)
  \leq  c_s \left(\gamma(1 + \d) + \min\left\{(\e0 \gamma\d)^{2/3}, \left[\gamma \e0b\d\left(1+\log\left(\frac{b}{\e0r_0}\right)\right)\right]^{1/2}\right\}\right).
\end{eqnarray*}
\end{Cor}
\begin{proof}
If $ (\e0 \gamma\d)^{2/3}\leq \left[\gamma \e0b\d\left(1+\log\left(\frac{b}{\e0r_0}\right)\right)\right]^{1/2}$, the assertion follows from Proposition \ref{prop:ub}(ii).\\
If $ (\e0 \gamma\d)^{2/3}\geq \left[\gamma \e0b\d\left(1+\log\left(\frac{b}{\e0r_0}\right)\right)\right]^{1/2}$
the assertion follows from Proposition \ref{prop:ub}(i).
\end{proof}
\section{Lower Bound}\label{sec:lowerbound}
\subsection{Preliminaries}
In this section we collect various results that will be needed in the proof of the lower bound of Theorem \ref{thm: main_intro}. 
We start with the following generalization of the isoperimetric estimate \eqref{eq:surface}, see Figure  \ref{fig:isoperineq2} for an illustration. 
\begin{Lemma}\label{isoperineq2}
Let $h: [0,1] \to [0,\infty)$ be Lipschitz with $h(0)=h(1) = 0$ and $\int_0^1 h \, d\mathcal{L}^1 = d$. Moreover, let $J\subseteq \Z$. For $i\in J$ let  $x_i\in (0,1)$ and $l_i>0$  be such that $x_{i} + l_i \leq x_j$ for all $i,j \in J$ with $i<j$ and $\bigcup_{i\in J} (x_i,x_i+l_i)\times (0,l_i) \subset \Omega_h$. Define $L_J\coloneqq \sum_{i\in J} l_i$ and 
$\d_J\coloneqq \sum_{i\in J} \int_{(x_i,x_i+l_i)} h\dd \calL^1 $. 
Then
\[
	\int_0^1 \sqrt{1+\vert h'\vert^2}\dd \calL^1 \geq 2 \frac{\d_J}{L_J}.
\]
\end{Lemma}
\begin{proof}
As $h$ is continuous, there exists $\overline{x}\in [0,1]$  such that $h(\overline{x})= 
\sup h=:\overline{h}.$ 
\begin{figure}[htpb!]
        \centering
        \includegraphics{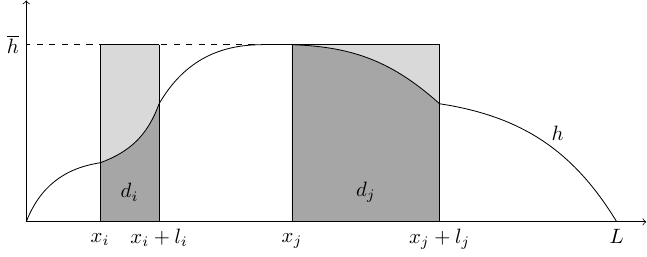}
        \caption{possible configuration for Lemma \ref{isoperineq2}}
\label{fig:isoperineq2}
\end{figure}
Then, using $L_J \overline{h} \geq \d_J$, we find
\begin{align*}
	\int_a^b \sqrt{1+\vert h'\vert^2} \dd \calL^1&\geq \int_a^b \vert h'\vert \dd \calL^1
		= \int_a^{\overline{x}} \vert h'\vert \dd \calL^1 + \int_{\overline{x}}^b \vert h'\vert \dd \calL^1
		\geq 2 \overline{h} 
		\geq 2 \frac{\d_J}{L_J}.
\end{align*} 
\end{proof}

We recall the Korn's inequality for fields with non-vanishing $\curl$, which follows from \cite[Theorem 11]{GLP}.

\begin{Thm}\label{thm: korn}
    Let $A \subseteq \R^2$ be open, connected, bounded with Lipschitz boundary, and denote by $\mathcal{M}(A;\R^2)$ the space of vector-valued Radon measures. Then there exists $c_A > 0$ such that for all $H \in L^2(A;\R^{2 \times 2})$ with $\curl H \in \mathcal{M}(A;\R^2)$ it holds
    \[
    \min_{W \in Skew(2)} \int_A |H - W|^2  \dd \calL^2 \leq c_A \left( \int_A |H_{\textup{sym}}|^2 \dd \calL^2 + |\curl H|(A)^2 \right),
    \]
    where $Skew(2)\subseteq\R^{2\times 2}$ denotes the space of skew-symmetric matrices.
\end{Thm}
\begin{Rm}\label{rem:Korn}
    Note that by scaling, the constant $c_A$ can be chosen to be uniform for all annuli with the same thickness ratio. Similarly, the constant can be chosen uniformly for all rectangles with side ratios between $\frac14$ and $4$.
\end{Rm}

Next, we state the following estimate for annuli which do not carry too much $\curl$ relative to the $\curl$ in the enclosed ball.

\begin{Lemma}\label{lemma: est annulus}
Let $R > r > 0$. Let $H \in L^2(B_R(0) ;\R^{2\times 2})$ with $\curl H = \burg\sum_{i=1}^N \delta_{x_i} \ast J_{r_0}$ 
such that $\left|\curl H\right|(B_R(0) \setminus B_r(0)) \leq \sqrt{\frac{\log(R/r)}{4\pi c_{\textup{Korn}}(R/r) }} \left|\curl H\right| (B_r(0))$, where $c_{\textup{Korn}}(R/r) = c_{B_R(0) \setminus B_r(0)}$ is the constant from Theorem \ref{thm: korn} (c.f.~Remark \ref{rem:Korn}). Then it holds
\[
\int_{B_R(0) \setminus B_r(0)} |H_{\textup{sym}}|^2 \dd \calL^2 \geq \frac1{4 \pi c_{\textup{Korn}}(R/r)} |\curl H| (B_r(0))^2 \log(R/r).
\]
\end{Lemma}
\begin{proof}
First, we apply Theorem \ref{thm: korn} to find $W \in Skew(2)$ such that
\begin{align}\label{eq: est korn circle}
&c_{\textrm{Korn}}(R/r) \int_{B_R(0) \setminus B_r(0)} |\Hsym|^2 \dd \calL^2 \\ \geq  &\int_{B_R(0) \setminus B_r(0)} |H - W|^2 \dd \calL^2  - c_{\textrm{Korn}}(R/r) |\curl H|(B_R(0) \setminus B_r(0))^2. \nonumber
\end{align}
Next, if $H$ is smooth we estimate similarly to \cite[Remark 3]{GLP} using Stokes' theorem 
\begin{align*} 
\int_{B_R(0) \setminus B_r(0)} |H - W|^2 \dd \calL^2 &\geq \int_{r}^R \int_{\partial B_t(0)} |H-W|^2 \dd \calH^1 \dd t \nonumber\\
&\geq \int_r^R \frac{1}{2\pi t} \left| \int_{\partial B_t(0)} (H-W)\cdot \tau \dd \calH^1 \right|^2 \dd t \nonumber\\
&\geq \int_r^R \frac{1}{2\pi t} \left| \curl H (B_t(0))\right|^2 \dd t \nonumber\\
&\geq \frac1{2\pi} \log(R/r) \left| \curl H \right|(B_r(0))^2. 
\end{align*}
Here we used that due to the specific form of $\curl H$ we have $|\curl H (B_t(0))| = |\curl H|(B_t(0)) \geq |\curl H|(B_r(0))$ for all $t \in (r,R)$.
Then the estimate for general $H$
\begin{equation} \label{eq: est stokes circle}
\int_{B_R(0) \setminus B_r(0)} |H - W|^2 \dd \calL^2 \geq \frac1{2\pi} \log(R/r) \left| \curl H \right|(B_r(0))^2
\end{equation}
follows by a standard mollification argument.
The assertion now follows by combining \eqref{eq: est korn circle} and \eqref{eq: est stokes circle} with the assumption $\left|\curl H\right|(B_R(0) \setminus B_r(0)) \leq \sqrt{\frac{\log(R/r)}{4\pi c_{\text{Korn}}(R/r) }} \left| \curl H\right| (B_r(0))$.
\end{proof}
Additionally, we recall the so-called ball construction as introduced for the analysis of vortices in the Ginzburg-Landau energy, see \cite{Jerrard99,Sandier98}. For an application to dislocations see, for example, \cite{alicandro-et-al:14,Ginster:19}. This will allow us to prove logarithmic lower bounds on the energy under mild assumptions on the maximal number of dislocations.
\begin{Lemma}[Ball-construction]\label{lemma: ball construction}
    Let $(B_{r_i}(p_i))_{i \in I} $ be a finite family of open balls in $\R^2$. Then for every $t>0$ there exists a finite family of open balls $(B_{r_i(t)}(p_i(t)))_{i \in I(t)}$ with pairwise disjoint closures such that the following properties hold: 
    \begin{enumerate}
        \item $\sum_{i \in I(t)} r_i(t) \leq e^t \sum_{i \in I} r_i$,
        \item $\bigcup_{i \in I} B_{r_i}(p_i) \subseteq \bigcup_{i \in I(t)} B_{r_i(t)}(p_i(t))$,
        \item for all $s \in (0, t]$ and $i \in I(s)$ there exists  a unique $j \in I(t)$ such that $
B_{e^{t-s}r_i(s)}(p_i(s)) \subseteq B_{r_j(t)}(p_j(t))$.
    \end{enumerate}
\end{Lemma}
\begin{proof}
    We sketch the proof for the convenience of the reader in Figure \ref{sketchballconstruction}. In order to construct $I(0)$ and $(B_{r_i(0)}(p_i(0)))_{i \in I(0)}$, we iterate the following construction. If $\overline{B_{r_i}(p_i)} \cap \overline{B_{r_j}(p_j)} \neq \emptyset$ for $i\neq j$ then set $\tilde{I} = I \setminus \{j\}$ and replace $B_{r_i}(p_i)$ by the ball $B_{\tilde{r}_i}(\tilde{p}_i)$, where
    \[
\tilde{r}_i = r_i + r_j \qquad \text{ and } \qquad \tilde{p}_i = \frac{r_i}{r_i + r_j} p_i + \frac{r_j}{r_i + r_j} p_j. 
    \]
    Then $B_{r_i}(p_i) \cup B_{r_j}(p_j) \subseteq B_{\tilde{r}_i}(\tilde{p}_i)$. This procedure terminates after finitely many steps and defines the index set $I(0)$ and the family $(B_{r_i(0)}(p_i(0)))_{i \in I(0)}$. The claimed properties of this family can be easily checked.
    Next, as long as this defines a family of open balls with pairwise disjoint closures we set for $t>0$ the index set $I(t) = I(0)$, the radii $r_i(t) = e^t r_i(0)$ and the centers $p_i(t) = p_i(0)$. For the first time $t>0$ such that the family of balls $(B_{e^t r_i(0)}(p_i(0)))_{i \in I(0)}$ does not have pairwise disjoint closures anymore, we perform a merging procedure that is similar to the construction for $t=0$. Precisely, for two balls $\overline{B_{e^t r_i(0)}(p_i(0))} \cap \overline{B_{e^t r_j(0)}(p_j(0))} \neq \emptyset$ for $i \neq j$ we set $\tilde{I}(t) = I(0) \setminus \{j\}$, 
    \[
    \tilde{r}_i(t) = e^t \left( r_i(0) + r_j(0) \right) \qquad \text{ and } \qquad \tilde{p}_i(t) = \frac{r_i(0)}{r_i(0) + r_j(0)} p_i(0) + \frac{r_j(0)}{r_i(0) + r_j(0)} p_j(0). 
    \]
    Again, if the family $(B_{\tilde{r}_i(t)}(\tilde{p}_i(t)))_{i \in \tilde{I}(t)}$ has pairwise disjoint closures then set $I(t) = \tilde{I}(t)$, $r_i(t) = \tilde{r}_i(t)$ and $p_i(t) = \tilde{p}_i(t)$, otherwise iterate this construction.
    Eventually, for $s > t$ we set again $I(s) = I(t)$, $r_i(s) = e^{s-t} r_i(t)$ and $p_i(s) = p_i(t)$ as long as the family $(B_{r_i(s)}(p_i(s)))_{i \in I(s)}$ has pairwise disjoint closures. For the first $s > t$ such that this is not true anymore, we construct $I(s)$, $r_i(s)$ and $p_i(s)$ through the same merging procedure as before. All claimed properties are easily checked.
\end{proof}
\begin{figure}
        \centering
        \subfloat[initial situation]{\includegraphics[width=50mm]{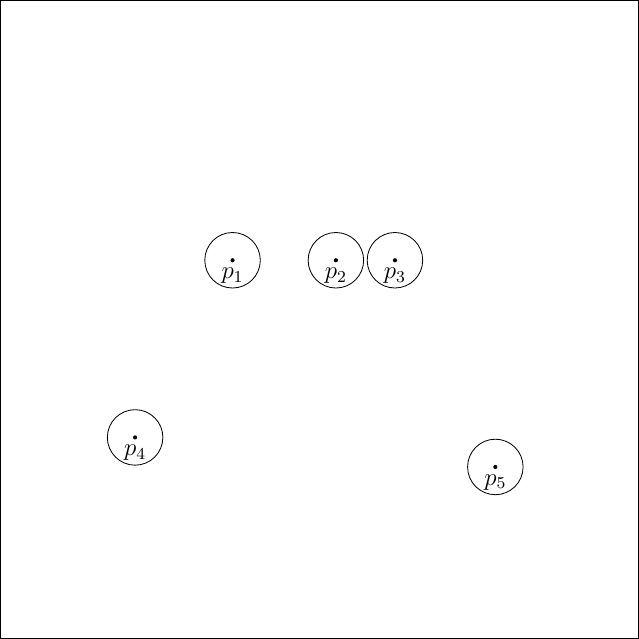}}\
        \subfloat[before merging]{\includegraphics[width=50mm]{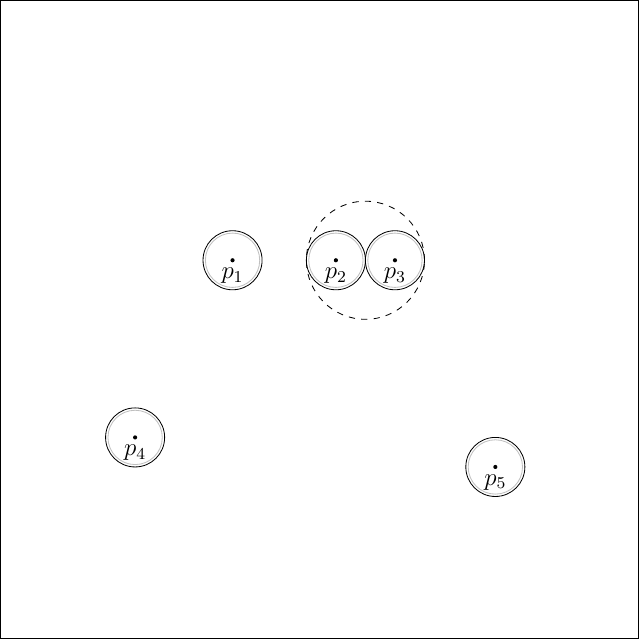}}\
        \subfloat[after merging]{\includegraphics[width=50mm]{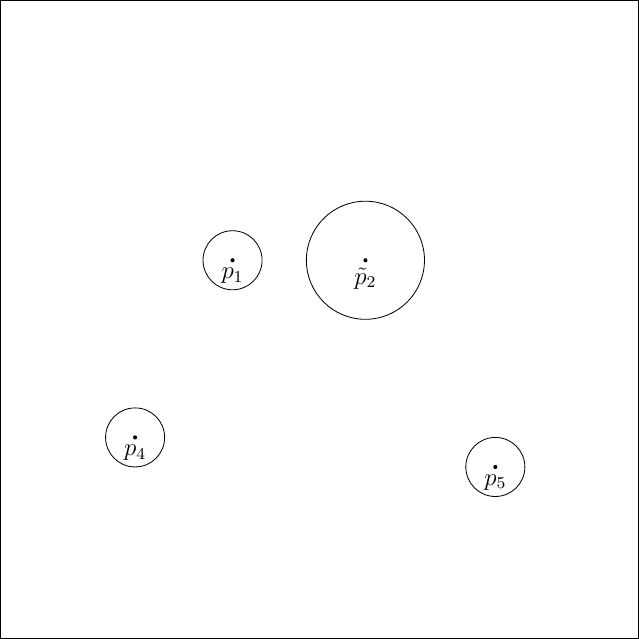}}\\
        \subfloat[before merging]{\includegraphics[width=50mm]{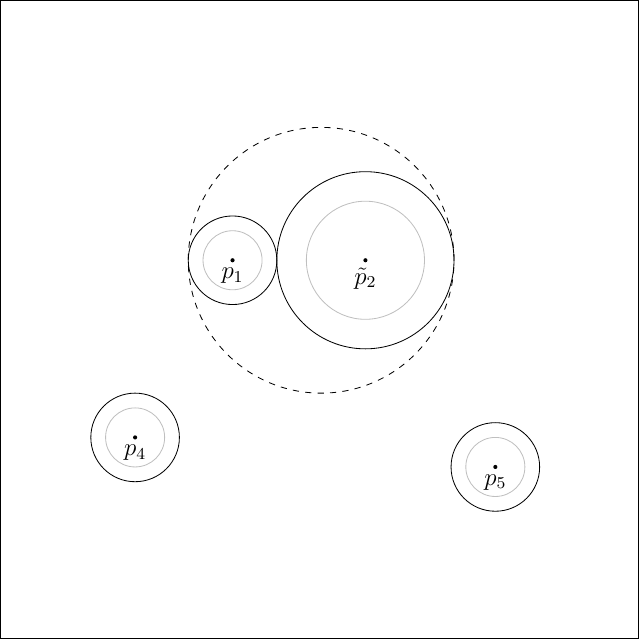}}\
        \subfloat[after merging]{\includegraphics[width=50mm]{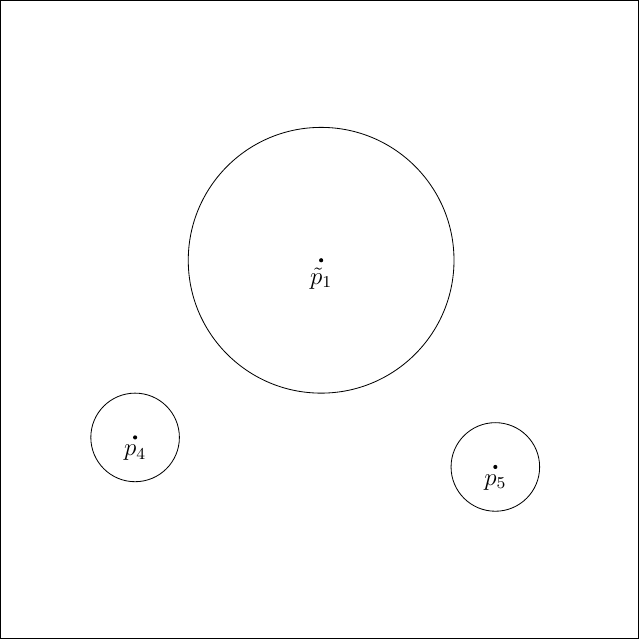}}\
        \subfloat[before merging, but another ball is overlapping]{\includegraphics[width=50mm]{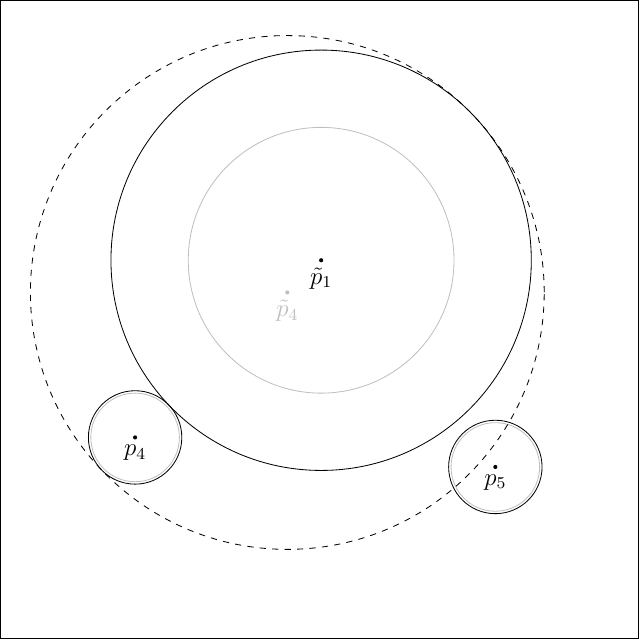}}\\
        \subfloat[before merging taking the overlapped ball into account]{\includegraphics[width=50mm]{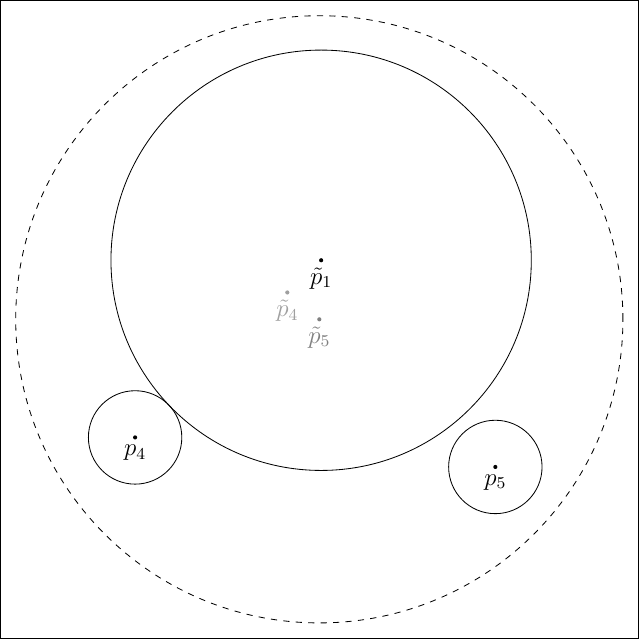}}\
        \subfloat[final situation]{\includegraphics[width=50mm]{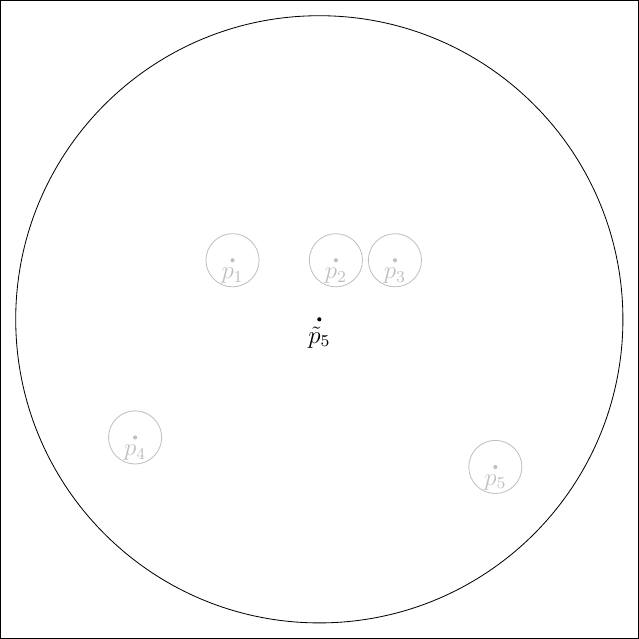}}
        \caption{sketch of ball construction for five balls with equal starting radii}
        \label{sketchballconstruction}
\end{figure}
\begin{Rm}\label{rem: ball construction}
    It can be seen in the construction that for $0 \leq s < t$ and $i \in I(t)$ it holds $i \in I(s)$ with $r_i(t) = e^{t-s} r_i(s)$ and $p_i(t) = p_{i}(s)$ if and only if the ball $B_{r_i(t)}(p_i(t))$ only includes the starting balls $B_{r_j}(p_j)$, $j \in I$, that are already included in  $B_{r_i(s)}(p_i(s))$.
\end{Rm}

\begin{Lemma} \label{lemma: log}
    There exists a constant $c > 0$ such that the following holds for all $b/\e0 > 64^4 r_0 > 0$. Let $A \subseteq \R^2$ open and $p_1, \dots, p_N \in \R^2$ with $N \leq \left\vert\log \left( \frac{b}{\e0 r_0} \right)\right\vert^2$. Additionally, let $H\in L^2(A;\R^{2\times 2})$ with $\up{curl\,}H= \burg (\sum_{n=1}^N J_{r_0} \ast \delta_{p_n})$ in $A$. Assume that $B_{b / (32\e0)}(p_n) \subseteq A$ for $n=1,\dots, n_0 \leq N$. Then
    \[
    \int_{\bigcup_{n=1}^{n_0} B_{b / (32 \e0)}(p_n) } \vert H_{\up{sym}}\vert^2 \,d \mathcal{L}^2 \geq c b^2 n_0 \, \log\left( \frac{b}{\e0r_0} \right).
    \]
\end{Lemma}
\begin{proof}
We apply the ball-construction from Lemma \ref{lemma: ball construction} to the family of open balls $(B_{r_0}(p_i))_{1 \leq i \leq N}$ to obtain for every $t\geq 0$ a family of open balls with pairwise disjoint closures $(B_{r_i(t)}(q_i(t)))_{i \in I(t)}$ satisfying the properties 1.,2.~and 3.~from Lemma \ref{lemma: ball construction}. 
Let us define $t_1 := \frac{\log\left( \frac{b}{\e0 r_0} \right)}{2} - 2 \log\left( \log(b / (\e0 r_0)) \right) > 0$. 
It follows from 1.~in Lemma \ref{lemma: ball construction} and $N \leq \log(b/ (\e0 r_0))^2$ that for all $i \in I(t_1)$ we have 
\begin{equation}\label{eq: est radius t1}
r_i(t_1) \leq e^{t_1} \log(b/ (\e0 r_0))^2 r_0 = \left( \frac{b}{\e0} r_0\right)^{1/2} \leq \frac1{64^2} \frac{b}{\e0} \leq \frac{1}{64} \frac{b}{\e0}. 
\end{equation}
From now on, let us set for $t\leq t_1$ 
\[
\tilde{I}(t) = \{ i \in I(t): \exists j \in I(t_1) \text{ ~s.~t.~ } B_{r_0}(p_n) \cup B_{r_i(t)}(q_i(t)) \subseteq  B_{r_j(t_1)}(q_j(t_1)) \text{ for some } 1 \leq n \leq n_0 \},
\]
the set of all balls at time $t$ that are related to the balls $B_{r_0}(p_n)$, $1\leq n \leq n_0$, at time $t_1$. Note that by \eqref{eq: est radius t1} it follows for every $i \in \tilde{I}(t)$ that 
\[
B_{r_i(t)}(q_i(t)) \subseteq B_{r_j(t_1)}(q_j(t_1)) \subseteq \bigcup_{n=1}^{n_0} B_{2 e^{t_1} \log(b/(\e0 r_0))^2 r_0}(p_n) \subseteq A.
\]

Now we distinguish two cases depending on whether a large amount of the mass of $\burg\sum_{k=1}^N \delta_{p_k}\ast J_{r_0}$ has accumulated in a single ball $B_{r_i(t_1)}(q_i(t_1))$, $i \in \tilde{I}(t_1)$, or not. If this is not the case, we derive lower bounds using a combinatorial argument  that guarantees long expansion times through the ball construction (see claim 2 below). If this is the case then we use Lemma \ref{lemma: est annulus} to obtain estimates in the domain $B_{b/(64 \e0 r_0)}(q_i(t_1)) \setminus B_{r_i(t_1)}(q_i(t_1)) \subseteq A$ (see claim 1 below).

For this, we fix $k_1>1$ such that $\log(k_1) \leq \frac{1}{4}$. Then 
\[
\left\lfloor \frac{ \frac12 \log \left( \frac{b}{\e0 r_0} \right) - \log(64) }{\log(k_1)}  \right\rfloor \geq \left\lfloor \frac{ \frac14 \log \left( \frac{b}{\e0 r_0} \right) }{\log(k_1)}  \right\rfloor \geq \log \left( \frac{b}{\e0 r_0}\right) -1 \geq \frac12 \log \left( \frac{b}{\e0 r_0}\right)   \geq 2. 
\]
Next, set $K = \lceil 16 \sqrt{\frac{4\pi c_{\textrm{Korn}}(k_1)}{\log(k_1)}} \rceil$, where $c_{\text{Korn}}(k_1)$ is the constant from Theorem \ref{thm: korn} for annuli with thickness ratio $k_1$ (c.f.~Remark \ref{rem:Korn}).
\newline 
\underline{Claim 1:} If there exists $i \in \tilde{I}(t_1)$ such that $\#\{ p_{n} \in B_{r_i(t_1)}(q_i(t_1)): 1 \leq n \leq n_0 \} \geq K \log \left( \frac{b}{\e0 r_0} \right)$ then it holds for a constant $c(k_1)>0$ that
\begin{equation}\label{eq: log case 1}
\int_{A} |\Hsym|^2 \dd \calL^2 \geq c(k_1) b^2 \log\left( \frac{b}{\e0 r_0} \right)^3.
\end{equation}

Fix $L = \left\lfloor \frac{ \frac12 \log \left( \frac{b}{\e0 r_0} \right) - \log(64) }{\log(k_1)} \right\rfloor \geq \frac12 \log \left( \frac{b}{\e0 r_0}\right)$.
Note that it follows that 
\[
k_1^L r_i(t_1) \leq k_1^L \left( \frac{b}{\e0} r_0 \right)^{1/2} \leq \frac{b}{64 \e0}.
\]
For $1 \leq l \leq L$, we define $A_l := B_{k_1^{l} r_i(t_1)}(q_i(t_1)) \setminus B_{k_1^{l-1} r_i(t_1)}(q_i(t_1))$.
Note that by definition of $\tilde{I}(t_1)$ there exists $1 \leq n \leq n_0$ such that $p_n \in A_l$. 
Hence, $A_l \subseteq B_{2 k_1^{l} r_i(t_1)}(p_n) \subseteq B_{b/(32 \e0)}(p_n) \subseteq A$. 
Next, let us assume that there exists $J \subseteq \{1,\dots, L\}$ with $\#J \geq \frac{L}2$ such that
\[
|\curl H|(A_l) \geq b \sqrt{\frac{\log(k_1)}{4\pi c_{\textrm{Korn}}(k_1)}}K \log\left( \frac{b}{\e0 r_0} \right).
\]
Then it holds that 
\[
b N \geq |\curl H|(A) \geq b \frac{L}{2} \sqrt{\frac{\log(k_1)}{4\pi c_{\textrm{Korn}}(k_1)}} K \log\left( \frac{b}{\e0 r_0} \right) \geq 8 b \log\left( \frac{b}{\e0 r_0} \right)^2,
\]
which contradicts $N \leq \log\left( \frac{b}{\e0 r_0} \right)^2$.
Hence, there exists $J \subseteq \{1,\dots, L\}$ with $\#J \geq \frac{L}2$ such that
\[
|\curl H|(A_l) \leq b \sqrt{\frac{\log(k_1)}{4\pi c_{\textrm{Korn}}(k_1)}}K \log\left( \frac{b}{\e0 r_0} \right).
\]
Hence, we obtain by Lemma \ref{lemma: est annulus} that
\begin{align*}
\int_{A} |\Hsym|^2 \dd \calL^2 & \geq \sum_{l \in J} \int_{A_l} |\Hsym|^2 \dd \calL^2 \\
&\geq \#J \frac1{4\pi c_{\textrm{Korn}}(k_1)} \left( |\curl H|(B_{r_i(t_1)}(q_i(t_1))) \right)^2  \log(k_1) \\
&\geq \frac{L}2 \frac1{4\pi c_{\textrm{Korn}}(k_1)} b^2 \left( K \log\left( \frac{b}{\e0 r_0} \right) \right)^2 \log(k_1) \\
&\geq \frac{\log(k_1)}{8\pi c_{\textrm{Korn}}(k_1)} b^2 K^2 \log\left( \frac{b}{\e0 r_0} \right)^3 
\end{align*}
This shows \eqref{eq: log case 1}.  

\underline{Claim 2:} If $i \in \tilde{I}(t_1)$ is such that $\#\{ p_{n} \in B_{r_i(t_1)}(q_i(t_1)): 1 \leq n \leq N \} < K \log \left( \frac{b}{\e0 r_0} \right)$ then we have for a constant $c(K)>0$ that
\begin{equation}\label{eq: log case 2}
\int_{B_{r_i(t_1)}(q_i(t_1))} |\Hsym|^2 \dd \calL^2 \geq c(K) b^2 \#\{ p_n \in B_{r_i(t_1)}(q_i(t_1)): 1 \leq n \leq n_0 \} \log \left( \frac{b}{\e0 r_0} \right) .
\end{equation}

Let us fix $M = 4 \lceil K \log( b / (\e0 r_0) ) \rceil$ and define $s_m = \frac{m}{M} \lfloor t_1  \rfloor$ for $0 \leq m \leq M$. 
Since $\#\{ p_n \in B_{r_i(t_1)}(q_i(t_1)): 1 \leq n \leq N \} < K \log \left( \frac{b}{\e0 r_0} \right)$ it follows by Remark \ref{rem: ball construction} that there exist $J \subseteq \{1,\dots, M \}$ with $\#J \geq \frac{M}{4}$ such that for all $m \in J$ it holds for all $j \in I(s_{m})$ with $B_{r_j(s_m)}(q_j(s_m)) \subseteq B_{r_i(t_1)}(q_i(t_1))$ that these balls are purely expanding between $s_m$ and $s_{m+1}$, namely $r_j(s_{m+1}) = e^{s_{m+1} - s_m} r_j(s_m)$ and $q_j(s_{m+1}) = q_j(s_m)$. 
Note that this implies in particular that $\curl H = 0$ in $B_{r_j(s_{m+1})}(q_j(s_{m+1})) \setminus B_{r_j(s_m)}(q_j(s_m))$.
It follows by Lemma \ref{lemma: est annulus} that
\begin{align*}
&\int_{B_{r_i(t_1)}(q_i(t_1))} |\Hsym|^2 \dd \calL^2 \\ 
\geq &\sum_{m \in J} \sum_{B_{r_j(s_m)}(q_j(s_m)) \subseteq B_{r_i(t_1)}(q_i(t_1))} \int_{B_{r_j(s_{m+1})}(q_j(s_{m+1})) \setminus B_{r_j(s_m)}(q_j(s_m)) } |\Hsym|^2 \dd \calL^2 \\
\geq &\sum_{m \in J} \sum_{B_{r_j(s_m)}(q_j(s_m)) \subseteq B_{r_i(t_1)}(q_i(t_1))}  \frac{\log(e^{s_1})}{4\pi c_{\textrm{Korn}}(e^{s_1})} b^2 \, (\#\{ p_n \in B_{r_j(s_m)}(q_j(s_m)): 1 \leq n \leq n_0 \})^2 \\
\geq &\sum_{m \in J} \sum_{B_{r_j(s_m)}(q_j(s_m)) \subseteq B_{r_i(t_1)}(q_i(t_1))}  \frac{\log(e^{s_1})}{4\pi c_{\textrm{Korn}}(e^{s_1})}  b^2 \, \#\{ p_n \in B_{r_j(s_m)}(q_j(s_m): 1\leq n \leq n_0 \} \\
= &(\# J) \frac{\log(e^{s_1})}{4\pi c_{\textrm{Korn}}(e^{s_1})} b^2 \, \#\{ p_n \in B_{r_i(t_1)}(q_i(t_1)): 1 \leq n \leq n_0 \} \\
\geq & \lceil K \log( b / (\e0 r_0) ) \rceil \frac{\log(e^{s_1})}{4\pi c_{\textrm{Korn}}(e^{s_1})} b^2 \, \#\{ p_n \in B_{r_i(t_1)}(q_i(t_1)): 1 \leq n \leq n_0 \} \\
\geq & c(K) b^2 \, \#\{ z_n \in B_{r_i(t_1)}(q_i(t_1)): 1 \leq n \leq n_0 \} \, \log\left( \frac{b}{\e0 r_0} \right).
\end{align*}
This shows \eqref{eq: log case 2}. \newline 

Now, note that if the assumption of claim 1 is true for one $i \in \tilde{I}(t_1)$, we find that 
\[
\int_{B_{b/(64\e0 r_0)}(p_i)} |\Hsym|^2 \dd \calL^2 \geq c(k_1) b^2 \log \left( \frac{b}{\e0 r_0} \right)^3 \geq c(k_1) b^2 n_0 \log \left( \frac{b}{\e0 r_0} \right),
\]
where we used that $n_0 \leq \log \left( \frac{b}{\e0 r_0}\right)^2$.
If, on the other hand, for all $i \in \tilde{I}(s_1)$ the assumption of claim 1 is not satisfied, we find using claim 2 and summing over all $i \in \tilde{I}(t_1)$
\[
\int_{\bigcup_{n=1}^{n_0} B_{b / (32\e0)}(p_n) } | H_{\textrm{sym}}|^2 \,d \mathcal{L}^2 \geq \sum_{i \in \tilde{I}(t_1)} \int_{B_{r_i(t_1)}(q_i(t_1))} |\Hsym|^2 \dd \calL^2 \geq c(K) b^2 n_0 \, \log\left( \frac{b}{\e0r_0} \right).
\]
\end{proof}

Complementing the result above we show here that the elastic energy (after application of Korn's inequality) can be estimated similarly to the fully elastic setting if $\curl H_1$ is small. 

\begin{Lemma}\label{eenergybetween2} Let $(h,H,\sigma) \in \adm, x_i \in (0,1)$ and $l_i\in (0,1-x_i)$ such that $(x_i,x_i+l_i)\times (0,l_i) \in \Omega_h$. Further let $\overline{x}\in (x_i,x_i+l_i/8)$ and assume $\int_{(\overline{x}, 2x_i+l_i-\overline{x}\times (0,l_i/2)} \curl H_1 \dd \calL^2 < \e0 l_i / 4.$ Then for any $W\in Skew(2)$ it follows that
\[
    \int_{(\overline{x}, 2x_i+l_i-\overline{x}\times (0,l_i/2)} \vert H-W\vert^2\dd \calL^2 \geq \frac{\e0^2 l_i^2}{768}.
\]
\end{Lemma}

\begin{proof}
We extend $H$ by $\begin{pmatrix}  \e0 && 0 \\ 0 && 0 \end{pmatrix}$ to $(x_i,x_i+l_i) \times [0,-\infty)$.
Next, set $A_1 \subseteq (\overline{x}, x_i+l_i/4)$ to be the set of all $x \in (\overline{x}, x_i+l_i/4)$ such that
\[
    \int_{\{x,2x_i+l_i-x\}\times (0,l_i/2)} \vert H-W\vert^2 \dd \calH^1 \leq \frac{16}{l_i} \int_{(\overline{x},2x_i+l_i-\overline{x})\times (0,l_i/2)} \vert H-W\vert^2 \dd \calL^2.
\]
Then $\mathcal{L}^1(A_1) \geq \frac{l_i}{16}$.
Similarly, set $A_2$ to be the set of all $y\in (l_i/4, l_i/2)$ such that
\[
    \int_{(\overline{x},2x_i+l_i-\overline{x})\times \{y\}} \vert H-W\vert^2 \dd \calH^1 \leq \frac{8}{l_i} \int_{(\overline{x},2x_i+l_i-\overline{x})\times (0,l_i/2)} \vert H-W\vert^2 \dd \calL^2.
\]
It follows $\mathcal{L}^1(A_2) \geq \frac{l_i}8$.
\begin{figure}[htpb!]
        \centering
        \includegraphics{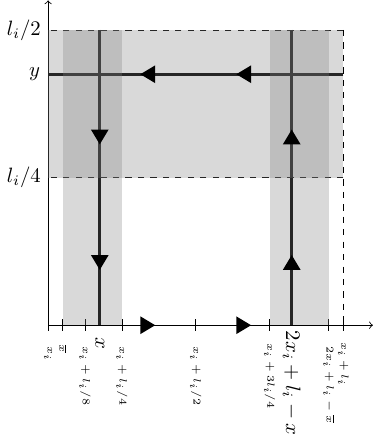}
        \caption{slice selection in Lemma \ref{eenergybetween2} with indicated path for Stokes' theorem}\label{fig: stokes}
\end{figure}
Next, let $  0 <  \varepsilon < \delta < (\overline{x}- x_i) < l_i / 8$ and $\varphi_{\varepsilon}$ a standard mollifier. 
Define $H^{\delta}(x,y) = H(x,y-\delta)$ and $H_{\varepsilon}^{\delta} = H^{\delta} \ast \varphi_{\varepsilon}$.
Note that $H^{\delta}_{\varepsilon} \stackrel{\varepsilon \to 0}{\to} H^{\delta} \stackrel{\delta \to 0}{\to} H$ in $L^2((\overline{x},2x_i + l_i - \overline{x}) \times (0,l_i/2))$ and hence by a diagonal argument there exists a sequence $(\varepsilon_{\delta})_{\delta}$ such that $H^{\delta}_{\varepsilon_{\delta}} \to H$ in $L^2((\overline{x},2x_i + l_i - \overline{x}) \times (0,l_i/2))$ as $\delta \to 0$.
Then by a Fubini argument it holds (up to a subsequence) for $\mathcal{L}^1$-almost all $x \in A_1$ and $y \in A_2$ that 
\begin{align*}
\lim_{\delta \to 0} \int_{\{x, 2x_i + l_i - x\} \times (0,l_i/2) } |H^{\delta}_{\varepsilon_{\delta}}-W|^2 \dd \calH^1 &= \int_{\{x, 2x_i + l_i - x\} \times (0,l_i/2)} |H-W|^2 \dd \calH^1 \\ &\leq \frac{16}{l_i} \int_{(\overline{x} ,2 x_i + l_i - \overline{x}) \times (0,l_i/2)} |H-W|^2 \dd \calL^2.
\end{align*}
and 
\begin{align*}
\lim_{\varepsilon \to 0} \int_{ (\overline{x} ,2 x_i + l_i - \overline{x})  \times \{y\} } |H_{\varepsilon_{\delta}}^{\delta}-W|^2 \dd \calH^1 &= \int_{ (\overline{x} ,2 x_i + l_i - \overline{x})  \times \{y\} } |H-W|^2 \dd \calH^1 \\ &\leq \frac{8}{l_i} \int_{(\overline{x} ,2 x_i + l_i - \overline{x}) \times (0,l_i/2)} |H-W|^2 \dd \calL^2.
\end{align*}
From now on, fix such a pair $x \in A_1$ and $y \in A_2$.
Then, we compute using Stokes' theorem (note that $H^{\delta}_{\varepsilon_{\delta}}$ is smooth and satisfies $\left(H^{\delta}_{\varepsilon_{\delta}} -W\right)_1(x,0) \cdot (1,0) = \e0$) and H\"older's inequality (c.f.~Figure~\ref{fig: stokes})
\begin{align*}
&\left\vert \int_{ (x,2x_i + l_i -\overline{x} ) \times (0,y)} \curl \left(H^{\delta}_{\varepsilon_{\delta}}-W\right)_1 \dd \calL^2 - (2x_i + l_i - 2 x) \e0\right\vert \nonumber \\
\leq &  \int_{\{x, 2x_i + l_i - x\} \times (0,l_i/2)} \vert H^{\delta}_{\varepsilon_{\delta}}-W\vert \dd \calH^1 + \int_{ (x, 2x_i + l_i - x) \times \{y\} } \vert H^{\delta}_{\varepsilon_{\delta}}-W\vert \dd \calH^1 \nonumber \\
\leq  &\sqrt{l_i} \left(\int_{\{x, 2x_i + l_i - x\} \times (0,l_i/2)} \vert H^{\delta}_{\varepsilon_{\delta}}-W\vert^2 \dd \calH^1\right)^{1/2} \\ & \qquad + \sqrt{2x_i+l_i-2x} \left( \int_{ (x, 2x_i + l_i - x) \times \{y\} } \vert H^{\delta}_{\varepsilon_{\delta}}-W\vert^2 \dd \calH^1 \right)^{1/2} 
\end{align*}
As $\curl \left(H^{\delta}_{\varepsilon_{\delta}}\right)_1 \to \curl H$ in $L^2_{\text{loc}}((x_i,x_i+l_i) \times (-\infty,l_i);\R^2)$ and $\curl W = 0$, we obtain as $\delta \to 0$ from the above estimate, the fact that $0 < (2x_i - l_i - 2x) \leq l_i$, and the choice of $x$ and $y$  that
\begin{align*}
    &\left\vert \int_{ (x,2x_i + l_i -x ) \times (0,y)} \curl H_1 \dd \calL^2 - (2x_i + l_i - 2 x) \e0\right\vert \\
    \leq & (4+\sqrt{8}) \left( \int_{(\overline{x} ,2 x_i + l_i - \overline{x}) \times (0,l_i/2)} \vert H-W\vert^2 \dd \calL^2 \right)^{1/2}.
\end{align*}
Eventually, note that $(2x_i + l_i - 2 x) \geq l_i/2$. Therefore we may estimate
\[
    \left\vert \int_{ (x,2x_i + l_i - x ) \times (0,y)} \curl \left(H^{\delta}_{\varepsilon_{\delta}}-W\right)_1 \dd \calL^2 - (2x_i + l_i - 2 x) \e0\right\vert \geq \frac{\e0 l_i}{4}.
\]
Thus, we obtain
\[
    \int_{(\overline{x},2x_i+l_i-\overline{x}) \times (0,l_i/2)} \vert H-W\vert^2 \dd \calL^2 \geq \frac{\e0^2l_i^2}{768},
\]
where we used that $16 \cdot (4+\sqrt{8})^2 \leq 768$.
\end{proof}

\subsection{Proof of the lower bound in Theorem \ref{thm: main_intro}}
In this section we prove the lower bound in Theorem \ref{thm: main_intro}. 
\begin{prop}
There is a constant $c >0$ and $\alpha \geq 64^4$ with the following property: For all 
$\gamma, \e0,b,\d>0$ and $r_0\in (0,1]$ with  $b/\e0 \geq \alpha r_0$ it holds
\begin{eqnarray*}
   c \, s(\gamma, e_0,b,d,r_0) 
  \leq \inf_{\adm} \calF(h,H,\sigma) 
\end{eqnarray*}
where $s(\gamma, e_0,b,d,r_0) =  \gamma(1 + \d) + \min\left\{\gamma^{2/3}\e0^{2/3} \d^{2/3}, \left[\gamma \e0b\d\left(1+\log\left(\frac{b}{\e0r_0}\right)\right)\right]^{1/2}\right\}$.    
\end{prop}
\begin{proof}
Let $(h,H,\sigma)\in \adm$. By \eqref{eq:Wgrowth} we will for simplicity assume that $W(H) = |\Hsym|^2$. \\

{\it Step 1: Estimate for connected $\Omega_h$. }\\
First, we assume that $\Omega_h$ is connected. 
For simplicity, set $\supp(h)=:[0,L]$. 
We will use the idea from the proof of the lower bound of \cite[Lemma 3.9]{BGZ:2015} to define local length scales. 
Note that here this choice is more involved due to the possibility of dislocations.

Fix $x_1\in [0,L]$ be such that $h(x_1)>0$.
Set 
\begin{equation}\label{def: lh}
    \ell_h \coloneqq \sup \{ l\in (0,1-x_1) \mid [x_1,x_1+l)\times (0,l)\subset \Omega_h\}
\end{equation} and 
\begin{equation}\label{def: ld}
    \ell_d \coloneqq \sup \{ l\in (0,1-x_1) \mid \# \left( \supp \sigma \cap [x_1,x_1+l)\times \R \right) \leq \left| \log \left( \frac{b}{\e0 r_0} \right) \right|^2\}.
\end{equation} 
Then define $l_1 = \min\{ \ell_h, \ell_d \}$.
Next, set $x_2\coloneqq x_1+l_1$ and repeat this process to iteratively define $(x_i)_{i=1}^{\infty}$ and $(l_i)_{i=1}^{\infty}$.
Moreover, define analogously
\begin{eqnarray*}
   l_0 \coloneqq \min\Big\{ &&\sup \left\{ l\in (0,x_1) \mid [x_1-l,x_1)\times (0,l)\subset \Omega_h\right\}, \\
   &&\sup \big\{ l\in (0,x_1) \mid \#  \left( \supp \sigma \cap (x_1-l,x_1)\times \R \right) \leq \left| \log\left( \frac{b}{\e0 r_0} \right) \right|^2\big\} \quad\Big\}, 
\end{eqnarray*}
and set $x_0 \coloneqq x_1 - l_0$.
Again, iterate this process to obtain the sequences $(x_i)_{i=0}^{-\infty}$ and $(l_i)_{i=0}^{-\infty}$.
Note that $\bigcup_{i=1}^{\infty} (x_{-i},x_i) = (0,L)$ since $h$ is continuous and $h(x) > 0$ for all $x \in (0,L)$ by the assumption that $\Omega_h$ is connected.
Next, define 
\[
\d_i\coloneqq \int_{[x_i,x_{i+1}]} h \dd \calL^1, \; E_i\coloneqq \int_{([x_i,x_{i+1}]\times \R_{>0})\cap \Omega_h} \vert H_{\up{sym}}\vert^2\dd \calL^2,  \text{ and }
 S_{i}\coloneqq \gamma \int_{[x_{i},x_{i+1}]} \sqrt{1+\vert h'\vert^2} \dd \calL^1.
\]
Additionally, define $N_i \coloneqq b^2 \# (\supp(\sigma) \cap [x_i,x_{i+1}] \times \R_{>0})$.
Then $2 \mathcal{F}(h,H,\sigma) \geq \sum_i S_i + E_i + N_i$.

We will now estimate the energy associated to $[x_i,x_{i+1}]\times \R_{>0} \cap \Omega_h$. For simplicity, we will assume $i \geq 1$. \\

\textit{Case 1: $l_i = \ell_d$ (in the sense that $l_i$ is determined through the analog of \eqref{def: ld}). } We will show that there exists a universal $c>0$ such that
\begin{equation}\label{eq: estimate case1}
E_i + N_i \geq c b \e0 l_i \log\left( \frac{b}{\e0 r_0} \right).
\end{equation}
For this, we distinguish two cases depending on the length of $l_i$. \\

\textit{Case 1a: $l_i \leq \frac{b}{\e0} \log(b/(\e0 r_0))$.} By the definition of $\ell_d$, we may estimate
\[
N_i = b^2 \# \left( \supp(\sigma) \cap [x_i,x_{i+1}] \times \R_{>0} \right) \geq b^2 \log\left( \frac{b}{\e0 r_0} \right)^2 \geq b \e0 l_i \log\left( \frac{b}{\e0 r_0} \right),
\]
which shows \eqref{eq: estimate case1} in this case. \\

\textit{Case 1b: $l_i \geq \frac{b}{\e0} \log(b/(\e0 r_0))$.} \\
\textit{Case 1b(i):} Let us assume that there exists $\overline{x}\in (x_i, x_i+l_i/8)$ such that
\[
    \int_{(\overline{x},2x_i+l_i - \overline{x}) \times (0,l_i/2)} \curl H_1 \dd \calL^2 <\sqrt{\frac{1}{2\cdot 768 \cdot c_{\textup{Korn}}}} \e0 l_i,
\]
where $c_{\textup{Korn}} > 1$ is Korn's constant for rectangles with side ration between $1/4$ and $4$ (c.f.~Remark \ref{rem:Korn}).
By the generalized Korn's inequality, Theorem \ref{thm: korn}, we obtain $W \in Skew(2)$ satisfying
\begin{align}
E_i &\geq \int_{(\overline{x},2x_i+l_i - \overline{x}) \times (0,l_i/2)} \vert H_{\textup{sym}}\vert^2 \dd \calL^2  \nonumber \\
    &\geq \frac{1}{c_{\textup{Korn}}} \int_{(\overline{x},2x_i+l_i - \overline{x}) \times (0,l_i/2)} \vert H-W\vert^2 \dd \calL^2 - \vert \curl H \vert ((\overline{x},2x_i+l_i - \overline{x}) \times (0,l_i/2))^2 \nonumber \\
    &\geq \frac{1}{c_{\textup{Korn}}} \int_{(\overline{x},2x_i+l_i - \overline{x}) \times (0,l_i/2)} \vert H-W\vert^2 \dd \calL^2 - \frac{\e0^2 l_i^2}{2\cdot 768\cdot c_{\textup{Korn}}}, \label{eq: 1b(i)1}
\end{align}
where we used that the specific form of $\curl H$ yields $\vert \curl H \vert ((\overline{x},2x_i+l_i - \overline{x}) \times (0,l_i/2)) = \int_{(\overline{x},2x_i+l_i - \overline{x}) \times (0,l_i/2)} \curl H_1 \dd \calL^2$.
Since $(2\cdot 768 \cdot c_{\textup{Korn}})^{-1/2} \leq 1/4$ we may now invoke Lemma \ref{eenergybetween2} so that
\begin{align} \label{eq: 1b(i)2}
    \frac1{c_{\textup{Korn}}} \int_{(\overline{x},2x_i+l_i - \overline{x}) \times (0,l_i/2)} \vert H-W\vert^2 \dd \calL^2 \geq \frac{\e0^2 l_i^2}{768 \cdot c_{\textup{Korn}}}.
\end{align}
Combining \eqref{eq: 1b(i)1}, \eqref{eq: 1b(i)2} and $l_i \geq \frac{b}{\e0} \log(b / (\e0 r_0))$ we find
\[
E_i \geq \frac{\e0^2 l_i^2}{2 \cdot 768 \cdot c_{\textup{Korn}}} \geq \frac{1}{768 \cdot c_{\textup{Korn}}} b \e0 l_i \log \left( \frac{b}{\e0 r_0} \right),
\]
which implies \eqref{eq: estimate case1} in this case. \\

\textit{Case 1b(ii):} Let us assume that for all $x\in (x_i, x_i+l_i/8)$ it holds
\begin{align}\label{ass 1biinew}
    \int_{(x,2x_i+l_i - x) \times (0,l_i/2)} \curl H_1 \dd \calL^2 \geq \sqrt{\frac{1}{2\cdot 768 \cdot c_{\textup{Korn}}}} \e0 l_i.
\end{align}
In particular, (\ref{ass 1biinew}) holds true for $\overline{x}\coloneqq x_i + l_i/16$. 
Next, set $A\coloneqq \supp(\sigma) \cap (\overline{x}-l_i/32, 2x_i+l_i - \overline{x} + l_i/32) \times (0, 17l_i/32)$. 
Since $r_0 < b/(32\e0) < l_i/32$ it follows that
\[
    b\cdot \# A\geq \sqrt{\frac{1}{2\cdot 768 \cdot c_{\textup{Korn}}}} \e0 l_i.
\]
Moreover, since $b/\e0 < l_i$ it follows for every $p \in A$ that $B_{b/ (32 \e0)}(p) \subseteq (x_i,x_{i+1}) \times (-l_i,l_i)$.
For a sketch see Figure~\ref{fig: ball construction applied}.
In order to apply Lemma \ref{lemma: log} note that by definition of $\ell_d$ it holds
\[
    \# \left( \supp(\sigma) \cap (x_i, x_i+l_i)\times (0,l_i)\right)\leq \log\left(\frac{b}{\e0r_0}\right)^2
\]
holds true.
Additionally, extend $H$ by $\begin{pmatrix} \e0 && 0 \\ 0 && 0\end{pmatrix}$ to $\supp(h) \times (-\infty,0]$.
Then it follows from Lemma \ref{lemma: log} that
\[
    \int_{\bigcup_{p \in A} B_{b/(32\e0)}(p)} \vert H_{\textup{sym}} \vert^2 \dd \calL^2 \geq cb^2 \log\left(\frac{b}{\e0r_0}\right) \#A.
\]
Consequently, we conclude
\begin{align*}
    E_i &\geq \int_{\bigcup_{p \in A} B_{b/(32\e0)}(p)} \vert H_{\textup{sym}} \vert^2 \dd \calL^2 - (\#A) \e0^2 \pi \left( \frac{b}{\e0} \right)^2\\
    &\geq cb^2 \log\left(\frac{b}{\e0r_0}\right) (\#A)- (\#A) \pi b^2\\
    &\geq (\#A)b^2 \left( c\log\left(\frac{b}{\e0r_0}\right) - \pi\right)\\
    &\geq \frac{c}{2}\sqrt{\frac{1}{2\cdot768\cdot c_{\textup{Korn}}}} \e0 l_i b \log\left(\frac{b}{\e0r_0}\right), 
\end{align*}
where we assume that $\alpha > 0$ is such that $\left( c\log\left(\frac{b}{\e0r_0}\right) - \pi\right) \geq c_2 \log\left( \frac{b}{\e0 r_0} \right)$.
This shows (\ref{eq: estimate case1}) and finishes the case $l_i = \ell_d$. \\
\begin{figure}[h!]
    \includegraphics[width=\textwidth]{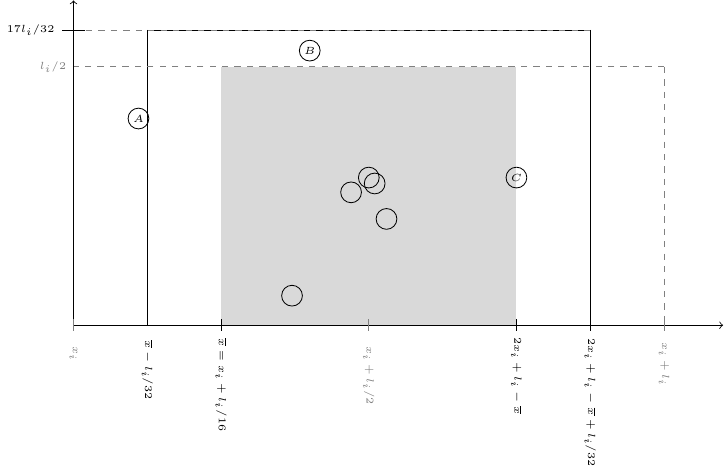}
    \caption{sketch for \textit{Case 1b(ii)}: Relatively much curl is found in the grey area. An estimate of the elastic energy is shown via a ball construction inside the black rectangle which represents the enlarged grey domain. By enlarging the grey domain three situations may appear: $(A)$ there is a new dislocation overlapping with the enlarged domain, $(B)$ a new dislocation is fully enclosed in the enlarged domain, $(C)$ a dislocation which was overlapping with the grey domain is now fully enclosed. The radii of the circles are $r_0$. Note that the sketch is not to scale. } \label{fig: ball construction applied}
\end{figure}

\textit{Case 2: $l_i = \ell_h$ (in the sense that $l_i$ is determined through the analog of \eqref{def: lh}).} In this case we will show that
\begin{equation}\label{eq: case2 elastic2}
 N_{i+1} + N_i + N_{i-1}  + E_i + S_i \geq c \left(\gamma  \e0 d_i\right)^{2/3}
\end{equation}
or
\begin{equation}\label{eq: case 2 log2}
    E_i \geq c b \e0 l_i \log\left( \frac{b}{\e0 r_0} \right). 
\end{equation}
We distinguish three cases depending on the length of $l_i$. \\

\textit{Case 2a: $l_i <\e0^{-2/3}\left(\gamma \d_i\right)^{1/3}$.} In this case, we can follow the argument from \cite{BGZ:2015}, which uses only surface energy. Note that by definition of $\ell_h$ it holds that $\min_{x_i \leq x \leq x_i+l_i} h(x) = \ell_h = l_i$. In addition, it holds $\max_{x_i \leq x \leq x_i + l_i} h(x) \geq \frac{\d_i}{l_i}$. Similarly to the proof of Lemma \ref{isoperineq2} we estimate
\begin{align}
	2 S_i \geq \gamma \int_{x_i}^{x_i+l_i} 1 + |h'(x)| \dd \calL^1 &\geq \gamma \left(l_i +  \max_{x_i \leq x \leq x_i + l_i} h(x) - \min_{x_i \leq x \leq x_i + l_i}h(x)\right) \nonumber \\ 
 &\geq \gamma \left(l_i + \frac{\d_i}{l_i} -l_i\right) = \frac{\gamma \d_i}{l_i} \geq \left(\gamma \e0 \d_i\right)^{2/3}, \nonumber
\end{align}
which shows \eqref{eq: case2 elastic2}.\\

\textit{Case 2b: $l_i  \geq \e0^{-2/3}\left(\gamma \d_i\right)^{1/3}$ and $l_i \geq b / \e0$.}
In this case, we will argue similarly to Case 1b. \\

\textit{Case 2b(i):} Let us assume that there exists $\overline{x}\in (x_i, x_i+l_i/8)$ such that
\[
    \int_{(\overline{x},2x_i+l_i - \overline{x}) \times (0,l_i/2)} \curl H_1 \dd \calL^2 <\sqrt{\frac{1}{2\cdot 768 \cdot c_{\textup{Korn}}}} \e0 l_i.
\]
Then we argue exactly as in case 1b(i) to find 
\begin{align*}
    E_i 
    \geq \frac{\e0^2 l_i^2}{2 \cdot 768 \cdot c_{\textup{Korn}}} &\geq \frac{1}{2 \cdot 768 \cdot c_{\textup{Korn}}}(\gamma \e0d_i)^{2/3}, 
\end{align*}
which shows \eqref{eq: case2 elastic2}. \\

\textit{Case 2b(ii):} Let us assume that for all $x\in (x_i, x_i+l_i/8)$ such that
\begin{align}\label{ass 1bii}
    \int_{(x,2x_i+l_i - x) \times (0,l_i/2)} \curl H_1 \dd \calL^2 \geq \sqrt{\frac{1}{2\cdot 768 \cdot c_{\textup{Korn}}}} \e0 l_i.
\end{align}
In particular, (\ref{ass 1bii}) holds true for $\overline{x}\coloneqq x_i + l_i/16$. 
Now, note that $l_i = \ell_h \leq \ell_d$ it follows that
\[
    \# \left( \supp(\sigma) \cap (x_i, x_i+l_i)\times (0,l_i)\right)\leq \log\left(\frac{b}{\e0r_0}\right)^2
\]
Then we argue exactly as in case 1b(ii) to find 
\[
E_i \geq \frac{c}{2} \sqrt{\frac{1}{2\cdot 768 c_{\textup{Korn}}}} \e0  l_i b \log \left(\frac{b}{\e0r_0}\right),
\]
which shows \eqref{eq: case 2 log2}. \\

\textit{Case 2c: $l_i  \geq \e0^{-2/3}\left(\gamma \d_i\right)^{1/3}$ and $l_i \leq b / \e0$.}
If $\int_{((x_i,x_i+l_i)\times \R_{>0}) \cap \Omega_h} \curl H\dd \calL^2 = 0$ we obtain by Korn's inequality and Lemma \ref{eenergybetween2} 
\begin{align*} 
	E_i \geq \frac{1}{c_{\text{Korn}}} \min_{W \in Skew(2)} \int_{(x_i,x_i + l_i) \times (0,l_i)} \vert H - W\vert^2 \dd \calL^2 \geq \frac{1}{768 \cdot c_{\text{Korn}}} \e0^2l_i^2 \geq \frac{1}{768\cdot c_{\text{Korn}}} \left(\gamma \e0 \d_i\right)^{2/3},
\end{align*}
which shows \eqref{eq: case2 elastic2} in this case.
If $\int_{((x_i,x_i+l_i)\times \R_{>0}) \cap \Omega_h} \curl H \dd \calL^2 \neq 0$, then there exists $p \in \supp(\sigma)$ such that $B_{r_0}(p) \cap ((x_i,x_i+l_i)\times \R_{>0}) \cap \Omega_h \neq \emptyset$ which implies that  $x_i - r_0 < p_1 < x_i + l_i + r_0$. 
If $x_i \leq p_1 \leq x_i+l_i$ then $N_i \geq b^2$. Let us now assume that $p_1 > x_i + l_i$. If $l_{i+1} = \ell_h$ then it follows that $l_{i+1} \geq p_1 - x_i - l_i$ since $h(x) \geq r_0$ for all $x_i + l_i \leq x \leq p_1$. It follows that $N_{i+1} \geq b^2$. A similar argument shows that if $l_{i-1} = \ell_h$ and $z_1 \leq x_i$ then $l_{i-1} \geq x_i - z$. Consequently, $N_{i-1} \geq b^2$. If $l_{i+1} = \ell_d$ or $l_{i-1} = \ell_d$ it follows by definition of $\ell_d$ that $N_{i-1} + N_{i+1} \geq b^2$. In summary, we obtain using $b \geq l_i \e0 \geq (\e0 \gamma d_i)^{1/3}$
\[
    N_{i-1} + N_i + N_{i+1} \geq b^2 \geq \left(\e0\gamma  d_i\right)^{2/3}.
\]
This shows \eqref{eq: case2 elastic2} which concludes Case $2$ where $l_i = \ell_h$. \\

Now, let $J_1\subset \N$ be the indices such that \eqref{eq: estimate case1} or (\ref{eq: case 2 log2}) holds, i.e. Case $1, 2b(ii)$, and $J_2\subset \N$ the indices such that \eqref{eq: case2 elastic2}, i.e. Cases $2a, 2b(i)$, and $2c$, holds.
Define $\d_{J_1}\coloneqq \sum_{i\in J_1} \d_i$, $\d_{J_2}\coloneqq \sum_{i\in J_2} \d_i$ as well as $L_{J_1}\coloneqq \sum_{i\in J_1} l_i$.
Overall we estimate using Lemma \ref{isoperineq2}, the subadditivity of the function $t \to t^{2/3}$, and minimizing in $L_{J_1}$
\begin{align*}
	6 \calF(h,H,\sigma) &\geq \gamma \int_0^1 \sqrt{1+\vert h'\vert^2}\dd \calL^1 + \sum_{i \in \Z} S_i + E_i + N_{i-1} + N_i + N_{i+1} \\
		&\geq c \left[\gamma \frac{\d_{J_1}}{L_{J_1}}  + \sum_{i\in J_1} l_i\e0b \log(b/(\e0r_0)) + \sum_{i\in J_2} \left(\e0\gamma \d_i\right)^{2/3}\right] \\
		&=c \left[ \gamma \frac{\d_{J_1}}{L_{J_1}} + L_{J_1}\e0b \log(b/(\e0r_0)) +\sum_{i\in J_2} \left(\e0\gamma \d_i\right)^{2/3}\right] \\
		&\geq c \left[ (\log(b/(\e0r_0)) \e0b\gamma \d_{J_1})^{1/2} + \left(\e0\gamma \d_{J_2}\right)^{2/3}\right] , \\
        &\geq c \min\{(\log(b/(\e0r_0)) \e0b\gamma \d_{J_1})^{1/2}, \left(\e0\gamma \d_{J_2}\right)^{2/3} \} \\
        &\geq \frac{c}{2^{1/2}} \min\{(\log(b/(\e0r_0)) \e0b\gamma \d)^{1/2}, \left(\e0\gamma \d\right)^{2/3} \} .
\end{align*} 
For the last inequality we used that $\d_{J_1} + \d_{J_2} = d$ and therefore $\d_{J_1} \geq \d/2$ or $\d_{J_2} \geq \d/2$. \\

\textit{Step 2: Estimate for arbitrary $\Omega_h$.}
By the continuity of $h$ the set $\Omega_h$ has at most countably many connected components $\Omega_j$ with volume $\d_j$. 
By applying Step 1 to every set $\Omega_j$ we find 
\begin{align*}
6 \calF(h,H,\sigma) &\geq \sum_j \frac{c}{2^{1/2}} \min\{(\log(b/(\e0r_0)) \e0b\gamma \d_j)^{1/2}, \left(\e0\gamma \d_j\right)^{2/3} \}.
\end{align*}
We split the sum in two parts corresponding to index sets  $I_1:=\{j: [\log(b/(\e0r_0)) \e0b\gamma \d_j]^{1/2}< \left(\e0\gamma \d_j\right)^{2/3}\}$ and $I_2:=\N\setminus I_1$, and note that $\sum_{j\in I_1} \d_j \geq  \d/2$ or $\sum_{j\in I_2} \d_j \geq  \d/2$. Therefore, 
using subadditivity, we obtain

\begin{align*}
6\calF(h,H,\sigma) 
&\geq \frac{c}{2} \min\{(\log(b/(\e0r_0)) \e0b\gamma \d)^{1/2}, \left(\e0\gamma \d\right)^{2/3} \}.
\end{align*}

\textit{Step 3: Conclusion.} Note that by Remark \ref{eq:surface}, we may estimate
\begin{align*}
	\calF(h,H,\sigma) \geq \gamma \left(\frac12 + \frac{\d}2\right).
\end{align*}
Combining this with Step 2 yields
\begin{align*}
\calF(h,H,\sigma) + 6 \calF(h,H,\sigma) &\geq \gamma \left(\frac12 + \frac \d2\right) + \frac{c}{2} \min\{(\log(b/(\e0r_0)) \e0b\gamma \d)^{1/2}, \left(\e0\gamma \d\right)^{2/3} \} \\
&\geq \min\left\{\frac12,\frac{c}{2}  \right\} s(\gamma, \e0,b,d,r_0). 
\end{align*}

\end{proof}
\section*{Acknowledgements}
Support of the {\it Deutsche Forschungsgemeinschaft}  under Germany's Excellence Strategy
{\it The Berlin Mathematics Research Center MATH+ and the Berlin Mathematical School
(BMS)} (EXC-2046/1, project 390685689) and within the Research
Training Group 2433 
(project number 384950143)  is gratefully acknowledged. BZ would like to thank Peter Bella and Michael Goldman for helpful discussions.

\bibliography{ScalingLawEpitaxialDislocations}
\bibliographystyle{siam}
\end{document}